\documentclass[10pt]{amsart}
\pdfoutput=1
\usepackage{txfonts}
\usepackage[T1]{fontenc}
\usepackage{fancyhdr}

\usepackage{amsmath,amsthm,amssymb,hyperref, mdframed}
\usepackage{mwe}
\usepackage{mathrsfs} 
\usepackage{tikz}
\usepackage[all]{xy}
\usetikzlibrary{decorations.markings}
\usetikzlibrary{backgrounds,shapes}
\usepackage{float, subfig}
\usepackage{tikz}
\usetikzlibrary{shapes,arrows}
\usetikzlibrary{shapes.geometric, arrows}
\usetikzlibrary{decorations.markings}
\usepackage{subfig}
\usepackage{tabularx}
\usepackage{enumitem}
\usepackage{amsmath,amsthm,amscd}
\usepackage{amssymb,amsfonts}
\usepackage{fancyhdr} 
\usepackage{pgf,tikz}
\usepackage{caption}
\usepackage{tikz-cd}
\usepackage{tikzscale}
\usetikzlibrary{patterns}
\usepackage{rotating}
\usepackage{xcolor}
\usepackage{lscape}
\usepackage{array} 
\usepackage{tabularx}
\usepackage{multirow}
\usepackage{multicol}
\usepackage{booktabs} 
\usepackage{caption}
\usepackage{xspace}
\usepackage{xfrac}
\usepackage{blindtext}
\usepackage[ruled,vlined]{algorithm2e}

\usetikzlibrary{decorations.markings}
\usetikzlibrary{patterns}
\usetikzlibrary{calc}
\usetikzlibrary{shapes.geometric}

\usepackage{color}
\usepackage{hyperref}
\hypersetup{
    colorlinks=true, 
    linktoc=all,     
    linkcolor=blue,  
    citecolor=blue,
    filecolor=blue,
    urlcolor=blue
}

\definecolor{aliceblue}{rgb}{0.9, 0.95, 1.0}

\usepackage{geometry}
\numberwithin{equation}{section}

\newcommand\Z{{\mathbb Z}}
\newcommand\ziz{{\Z[\,{\rm i}}\,]}

\newcommand{\C}{{\mathbb C}}

\newcommand{\pto}[1]{\small{\textsc{#1}}}

\newcommand{\spz}{\mathrm{Sp}(2g,\Z)}
\newcommand{\glplus}{\mathrm{GL}^+(2,\mathbb{R})}

\newcommand{\modul}{\textnormal{Mod}(S_{g})}

\newcommand{\rhomolzp}{\text{H}_1(X\setminus P(\omega),\, Z,\,\mathbb{Z})}
\newcommand{\ahomolz}{\text{H}_1(\,X,\,\mathbb{Z}\,)}
\newcommand{\rhomolz}{\text{H}_1(\,X,\, Z,\,\mathbb{Z}\,)}

\newcommand{\shomolzoo}{\textnormal{H}_1(S_{1,1},\mathbb{Z})}

\newcommand{\shomolz}{\textnormal{H}_1(S_{g},\mathbb{Z})}
\newcommand{\hm}{\mathcal H_g(\,\mu\,)}
\newcommand{\chia}{\chi_{a}}

\theoremstyle{plain}                    
\newtheorem{thm}{Theorem}[section]

\newtheorem{thma}{Theorem}

\newtheorem{lem}[thm]{Lemma}
\newtheorem{prop}[thm]{Proposition}
\newtheorem{cor}[thm]{Corollary}

\theoremstyle{definition}
\newtheorem{defn}[thm]{Definition}

\newtheorem{ex}[thm]{Example}

\newtheorem{rmk}[thm]{Remark}

\tikzstyle{rb} = [rectangle, rounded corners, minimum width=3cm, minimum height=1cm, text width=3cm, text centered, draw=black, fill=blue!30]
\tikzstyle{sb} = [rectangle, minimum width=2cm, minimum height=1cm, text width=3cm, text centered, draw=black, fill=violet!30]

\title[Relative period realization of holomorphic differentials with prescribed invariants]{Relative period realization of holomorphic differentials with prescribed invariants}

\author{Dawei Chen}
\address[Dawei Chen]{Department of Mathematics, Boston College, Chestnut Hill, MA 02467, USA}
\email{dawei.chen@bc.edu}

\thanks{Research of D.C. is supported in part by National Science Foundation Grant DMS-2301030, Simons  Travel Support for Mathematicians, and Simons Fellowship.}

\author{Gianluca Faraco}
\address[Gianluca Faraco]{Dipartimento di Matematica e Applicazioni U5, Universita` degli Studi di Milano-Bicocca, Via Cozzi 55, 20125 Milano, Italy}
\email{gianluca.faraco@unimib.it}

\thanks{Research of G.F. is partially supported by GNSAGA, a division of INdAM.}

\date{\today}

\begin{document}

\dedicatory{}

\begin{abstract}
We provide a complete description of realizable relative period representations for holomorphic differentials on Riemann surfaces with prescribed orders of zeros and additional invariants given by the hyperelliptic structure and spin parity. This answers a question posed by Simion Filip. 
\end{abstract}

\maketitle
\tableofcontents

\section{Introduction}

\noindent Let $\mu = (a_o, \ldots, a_k)$ be a signature of holomorphic differentials, where $a_i\geq 0$, and we allow $a_i = 0$ to record a marked ordinary point as a zero of order zero if necessary. Let $\hm$ denote the stratum of pairs $(X,\omega)$ where $X\in\mathcal M_g$ is a compact Riemann surface and $\omega \in\Omega(\,X\,)$ is a holomorphic differential on $X$. Every such a pair yields a representation 
\begin{equation}\label{eq:absrep}
    \chi_a\colon\ahomolz\longrightarrow \mathbb{C} \,\, \text{ such that } \,\, \gamma\longmapsto\int_\gamma \omega,
\end{equation}
known as the period character. Throughout the present paper, we refer to $\chi_a$ as the \textit{absolute period} representation. If we fix a set $Z\subset (X,\omega)$ of marked points, we define the \textit{relative period} representation as the map
\begin{equation}\label{eq:relrep}
    \chi\colon \rhomolz\longrightarrow \mathbb{C}.
\end{equation}

\noindent In \cite{OH} Haupt provided necessary and sufficient conditions for a representation $\chia\colon\ahomolz\longrightarrow \mathbb C$ to arise as the period character of some pair $(X,\omega)$; in \S\ref{sssec:vol} we refer to them as \textit{Haupt's conditions}. The same result has been subsequently rediscovered by Kapovich in \cite{KM2} by using Ratner's theory. Realizing a representation as a character in a prescribed stratum turns out to be a more subtle problem because, in the realizing process, the orders of zeros can no longer be ignored. In the same spirit of \cite{KM2}, Le Fils described in \cite{fils} necessary and sufficient conditions for a representation $\chi$ to be a character in a given stratum. Around the same time, Bainbridge--Johnson--Judge--Park in \cite{BJJP}  provided, with an independent and alternative approach, necessary and sufficient conditions for a representation to be realized in a connected component of a prescribed stratum. In the present paper, we aim to extend the study to the realization of relative periods for holomorphic differentials. 

\smallskip

\noindent Given a closed surface $S_{g}$ and a finite set of marked points $Z=\{z_o,\dots,z_k\}\subset S_{g}$, we shall consider a representation \(\chi\colon\textnormal{H}_1(\,S_g,\,Z,\,\mathbb Z\,)\longrightarrow \mathbb{C},\)
and wonder if there is a compact Riemann surface $X\in\mathcal M_g$ and a differential $\omega\in\Omega(\,X\,)$ with $k+1$ zeros at \(Z\) with relative period character $\chi$. In particular, for any given signature \(\mu\) 
we wonder if there exists a differential $\omega$ on $X$ with zeros of order $a_i$ at $z_i$ such that the period representation of $\omega$ is $\chi$. In this case we shall say that $\chi$ is a realizable relative period representation in the stratum $\hm$. See \S\ref{sec:absrelper} for a more detailed explanation of our problem. The main result of our paper is the following

\begin{thma}\label{thm:mainthm}
    Let \(\chi\colon\textnormal{H}_1(S_{g},\,Z,\,\mathbb Z)\longrightarrow \mathbb{C}\) be a relative period representation with absolute period representation \(\chi_a\). Let \(\mu\) be any signature of holomorphic differential. If \(\chi_a\) can be realized in \(\hm\) then \(\chi\) can be realized in every connected component of the same stratum.
\end{thma}

\noindent In other words, if an absolute period representation \(\chia\) can be realized as the absolute period character of some pair \((X,\omega)\) in a given stratum, then we can also find a holomorphic differential with the desired relative periods in the same stratum.

\smallskip

\noindent Upon the completion of this work, 
we learned that the main result in Theorem~\ref{thm:mainthm} has also been obtained by Le Fils independently in \cite{filsrel} with an approach that shares similar ideas but is developed in a different manner from ours. 

\subsection{From the complex-analytic perspective to geometry}\label{ssec:gpovts} 

\noindent The period realization problem can be formulated in complex-analytic terms. Assume \( g \geq 2 \) and let \( \chia\colon\textnormal{H}_1(S_{g},\,\mathbb Z)\longrightarrow \mathbb{C} \) be an absolute period representation. 
Finding a pair \( (X, \omega) \in \mathcal{H}_g(\,\mu\,) \) with period character \( \chi \) amounts to determine a holomorphic multivalued function \(F\colon X\longrightarrow \mathbb C\) arising as a solution of the following first-order linear ODE on \(X\):
\begin{equation}\label{eq:linode}
    dF=\omega
\end{equation}
such that its monodromy representation, say \( \rho \colon \pi_1(\,S_g\,) \to \mathbb{C} \), reduces to the representation \( \chia \) in homology. A solution \(F\) of the differential equation \eqref{eq:linode} lifts to a mapping \( \textnormal{dev} \colon \mathbb{H} \longrightarrow \mathbb{C} \) that satisfies the following conditions:
\begin{itemize}
    \item[1.] \( \textnormal{dev} \) is locally univalent outside a discrete set that corresponds to the lift of the set of zeros of \(\omega\), and
    \item[2.] \( \textnormal{dev} \) is equivariant with respect to a representation \( \rho \colon \pi_1(\,S_g\,) \longrightarrow \mathbb{C} \) in the sense that
    \begin{equation}
        \textnormal{dev}(\,\gamma\,\cdot\,z\,)\,=\,\rho(\,\gamma\,)\,+\,\textnormal{dev}(\,z\,)
    \end{equation}
    and that reduces to \(\chia\) in homology. 
\end{itemize}
We define this kind of functions as \textit{developing maps} and the reason of such a terminology will become clear later on. Vice versa, every holomorphic \(\rho-\)equivariant map \( \textnormal{dev} \colon \mathbb{H} \longrightarrow \mathbb{C} \) which is locally univalent outside an invariant discrete set yields a pair \((X,\omega)\). By Uniformization Theorem, \(X \cong \mathbb H/\pi_1(\,S_g\,)\) and \(\omega\in\Omega(\,X\,)\) is defined as the projection of the abelian differential \(\textnormal{dev}^*dz\) on \(\mathbb H\) via the covering projection. Thus finding a pair \( (X, \omega) \in \mathcal{H}_g(\,\mu\,) \) with prescribed period character \( \chi \) amounts to determining a map \( \textnormal{dev} \colon \mathbb{H} \longrightarrow \mathbb{C} \) that satisfies the properties above.

\smallskip

\noindent This perspective allows us to transform a problem of complex-analytic nature into one with a more geometric flavor. Indeed, such a developing map determines an atlas of charts with values in \( \mathbb{C} \) such that the transition maps are translations. In other words, a developing map induces a branched geometric structure locally modeled on \( (\mathbb{C}, \mathbb{C}) \), known as a \textit{translation surface}. On the other hand, every branched \( (\mathbb{C}, \mathbb{C}) \) structure on a closed surface \(S_g\) yields a developing map satisfying the properties above. See \S\ref{ssec:transurf} for more details. In turn, every such an atlas determines a triangulation of the surface itself, consisting of triangles that are isometric to Euclidean triangles. Some of these triangles can be combined to form polygons in the complex plane. Thus, we can state that a translation structure is always obtained by gluing a finite number of polygons in the plane via translations, identifying edges with other edges and vertices with other vertices. In the present paper, we shall adopt this approach. In fact, we shall construct pairs \( (X, \omega) \) with a prescribed structure by gluing polygons, typically squares and rectangles, see \S\ref{sec:sqrspesper}. Finally, we observe that, so far, we have merely reformulated the original problem in geometric terms without addressing the prescription of certain invariants of interest. For example, this description does not take into account the number and orders of the zeros of the differential. Similarly, it does not specify additional hyperelliptic and spin structures. In this sense, we will revisit these invariants from a geometric perspective.



\subsection{The period map} Let \(\mathcal H_g(\,\mu\,)\) be a stratum of holomorphic differentials. The so-called \textit{period map} defined as the association that maps an abelian differential $\omega$ to its period character, see \eqref{permap}. Unfortunately, such a map is not globally well-defined on \(\mathcal H_{g}(\,\mu\,)\), see the discussion in \cite[Section \S4.1.1]{NS}. In order to have a well-defined period map, we need to consider a suitable covering space that we denote by $\Omega\mathcal S_{g}(\,\mu\,)$.

\smallskip

\noindent Fix a reference closed connected oriented surface \(S_g\) of genus \(g\ge2\) and a set \(Z\) of \(k+1\) marked points. For any pair \((X,\omega)\in\mathcal H_g(\,\mu\,)\), a \textit{marking in homology} is an isomorphism \(m\colon \textnormal{H}_1(\,S_g,\,Z,\,\mathbb Z\,)\longrightarrow \textnormal{H}_1(\,X,\,Z(\,\omega\,),\,\mathbb Z\,)\). Here two triples \((X,\omega,\,m_X)\) and \((Y,\xi,\,m_Y)\) are equivalent, if there exists a biholomorphism of marked Riemann surfaces, say \(f\colon X \longrightarrow Y\), such that \(\omega=f^*\xi\) and \(m_Y=f_*\circ m_X\). A stratum of \textit{homologically marked translation surfaces} is defined as the space of equivalent classes of translation surfaces marked in homology and denoted by \(\Omega\mathcal S_g(\,\mu\,)\). The mapping class group \(\modul\) acts on \(\mathcal S_{g}\) by precomposition on the marking and such an action yields a covering map \(\Omega\mathcal S_{g}(\,\mu\,)\longrightarrow \mathcal H_{g}(\,\mu\,)\)
with covering group \(\spz\) by construction. The relative period map is then defined as the association
\begin{equation}\label{permap}
    \text{relPer}\colon\Omega\mathcal{S}_{g}(\,\mu\,)\longrightarrow \textnormal{H}^1\big(\,S_g,\,Z,\,\mathbb C\,\big)\cong\text{Hom}\,\Big(\,\textnormal{H}_1(\,S_g,\,Z,\,\mathbb Z\,),\,\mathbb{C}\,\Big)
\end{equation}
that maps a marked translation surface \((X,\omega,m)\) to its relative period character. It is well-known that this map is holomorphic and locally injective. As a consequence, the unmarked stratum \(\hm\) is also locally modeled on \(\textnormal{H}^1\big(\,S_g,\,Z,\,\mathbb C\,\big)\). 

\smallskip

\noindent Following \cite[Section \S3.1.3]{FS}, there is a natural short exact sequence induced by the long exact sequence of the pair \((\,S_g,\,Z\,)\), that is
\begin{equation}
    0 \longrightarrow \textnormal{H}_1\big(\,S_g,\,\mathbb Z\,\big) \longrightarrow \textnormal{H}_1\big(\,S_g,\,Z,\,\mathbb Z\,\big) \longrightarrow \widetilde{\textnormal{H}}_o\big(\,Z,\,\mathbb Z\,\big) \longrightarrow 0
\end{equation}
where \(\widetilde{\textnormal{H}}_o\big(\,Z,\,\mathbb Z\,\big)\) denotes  the reduced \(0-\)th homology of \(Z\). The dual sequence in cohomology with complex coefficients yields a surjection \(\textnormal{abs}\colon\textnormal{H}^1\big(\,S_g,\,Z,\,\mathbb C\,\big)\longrightarrow \textnormal{H}^1\big(\,S_g,\,\mathbb C\,\big)\) and the composition map \(\textnormal{abs}\circ\textnormal{relPer}\) yields an association whose image is the set of all absolute period representations that arise as the absolute part \(\chia\) of some relative period representation \(\chi\). 

\smallskip

\noindent In the same way, it is possible to define the absolute period map, which assigns to a marked translation surface its absolute period character:
\begin{equation}\label{abspermap}
    \text{Per}\colon\Omega\mathcal{S}_{g}(\,\mu\,)\longrightarrow \textnormal{H}^1\big(\,S_g,\,\mathbb C\,\big)
\end{equation}
\noindent whose image has been completely determined by \cite{fils} and \cite{BJJP}. In principle, the following inclusion holds: \(\textnormal{Im}\big(\,\textnormal{abs}\circ\textnormal{relPer}\,\big)\subseteq\textnormal{Im}\big(\,\textnormal{Per}\,\big)\). Our main Theorem \ref{thm:mainthm} states that, if we forget the marking in homology, then
\begin{equation}
    \textnormal{Im}\big(\,\textnormal{abs}\circ\textnormal{relPer}\,\big)\,=\,\textnormal{Im}\big(\,\textnormal{Per}\,\big)
\end{equation}
for every stratum \(\hm\), which thus provides a complete answer to the question posed by Simion Filip \cite[Question 3.1.15]{FS}. 

\medskip

\textit{A short digression on the fibres of the period map.} In the past decade, understanding the fibers of the absolute period map has been the subject of intense study from various perspectives. This line of research, initiated by McMullen in \cite{McM} and \cite{McM2} for surfaces of genus \( g = 2,3 \), was later extended by Hamenst\"adt in \cite{UH} and independently by \cite{CDF2}. Both works prove the ergodicity of the isoperiodic foliation for the principal stratum \(\mathcal H_g(\,1,\dots,1\,)\) for every genus \(g\ge2\). More recently, the study of these fibers was further developed by \cite{KW} and, finally, by Chaika--Weiss in \cite{CW}. We note that, since the relative period map \eqref{permap} is locally injective, the fibers are discrete.

\subsection{The strategy of proving the main Theorem} Our proof builds on the previous works of Le Fils \cite{fils} and Bainbridge--Johnson--Judge--Park \cite{BJJP}. 
We will primarily distinguish between two cases, depending on whether the absolute period character is discrete or not. The latter case is the simpler one to handle. From the aforementioned works, we know that if a representation is not discrete, it can be realized in every connected component of any stratum. In particular, it can be realized in every connected component of the minimal stratum \(\mathcal H_g(\,2g-2\,)\).  Based on this assumption, we will show that zeros can be split using a well-known surgery technique from the literature, which will be described in detail for completeness in \S\ref{sec:surgeries}. Using an inductive argument, we will show that the same representation can be realized in any other stratum while preserving, when necessary, additional structures such as the spin parity and the hyperelliptic structure. In this case, our proof does not directly construct the structures using polygons, as hinted in \S\ref{ssec:gpovts}. On the other hand, if the representation is discrete, in \S\ref{sec:sqrspesper}, we will construct square-tiled surfaces that realize the prescribed zeros in every connected component of every stratum. As we will see, it will be necessary to introduce a stronger necessary condition than the one imposed by condition \eqref{eq:hcond2s} below, and this stronger necessary condition will turn out to be sufficient as well. In this case, the realization is direct, and most of the current work consists of constructing these structures.

\subsection{Meromorphic differentials} Following the extensive study of holomorphic differentials on Riemann surfaces, the investigation of meromorphic differentials has also gained significant attention in recent years. A meromorphic differential \(\omega\) on a compact Riemann surface \(X\in\mathcal M_g\) yields an absolute period representation \(\chia\colon \textnormal{H}_1(\,X\setminus\,P(\,\omega\,),\,\mathbb Z\,)\longrightarrow \mathbb C\), where \(P(\,\omega\,)\) is the set of poles of \(\omega\). It is not difficult to show that each such representation arises as the period character of a meromorphic differential on a given Riemann surface, see \cite{CFG}. In other words, unlike the holomorphic case, there are no obstructions to the realization of a given representation. Less obvious is the realization of such a representation in a given stratum or connected component of a stratum. Both issues have been thoroughly addressed in the works \cite{CFG} and \cite{CF}. In \cite{fargup}, the authors provide an alternative version in which all singularities are at the punctures. Even in the meromorphic case, it is possible to define an absolute period map, and the results of these works fully describe the image when restricted to a stratum or its connected component. The study of the isoperiodic fibration is still an open problem, and only some special cases in this direction are currently known, see \cite{CD} and \cite{FTZ}. As in the holomorphic case, it is possible to define a relative homology representation, that is, a representation \(\chi\colon\rhomolzp\longrightarrow \mathbb C\). 
We plan to determine which relative homology representations can arise from meromorphic differentials in a given stratum in a future work. Indeed, the method we use in Section \S\ref{sec:notdis} can directly extend to the case of meromorphic differentials, and thus, realizing discrete representations will be the major challenge for the meromorphic case. 




\subsection{Structure of the paper} The present paper is organized as follows. In \S\ref{sec:absrelper} we review the  definition and properties of translation surfaces as well as their volumes, hyperelliptic structure, and spin structure. In \S\ref{sec:surgeries} we recall some topological surgeries of breaking up a zero and the movements of branched points which were systematically used in \cite{CDF} and \cite{CDF2}. In \S\ref{sec:notdis} we prove Theorem \ref{thm:mainthm} for the case of non-discrete representations. Finally, in \S\ref{sec:sqrspesper} we prove Theorem \ref{thm:mainthm} for the remaining case of discrete representations. In particular, we will utilize the realization of square-tiled surfaces with prescribed invariants.

\subsection*{Acknowledgments} D.C. would like to thank the organizers Viveka Erlandsson, Scott Mullane, and Paul Norbury of the MATRIX program "Teichm\"uller Theory and Flat Structures" in Creswick for their invitation and hospitality. G.F. would like to thank Guillaume Tahar and BIMSA in Beijing for their hospitality. The authors are also grateful to Thomas Le Fils for sharing his result and for the inspiring discussion upon the completion of the current paper. 

\medskip

\section{Preliminaries of translation surfaces and period representations}
\label{sec:absrelper}

\subsection{Strata of differentials}\label{ssec:strata} Let $\Omega\mathcal M_g$ be the Hodge bundle over $\mathcal M_g$ whose fiber over a compact Riemann surface $X$ is the space $\Omega(\,X\,)=\textnormal{H}^o(\,X,\, K_X\,)$ of holomorphic differentials on $X$. The moduli space $\Omega\mathcal M_g$ admits a natural stratification given by unordered partitions of $2g-2$. Let $\mu=(a_o,\dots,a_k)$ be any partition of $2g-2$. If an abelian differential $\omega$ has exactly $k+1$ zeros of respective orders $a_o,\dots,a_k$, then we say that $\omega$ is a holomorphic differential of type $\mu$.  As already alluded in the introduction, we denote a stratum of type $\mu$ as $\mathcal H_g(\mu)$. Therefore, a stratum parameterizes equivalent classes of $(X, \omega)$ up to biholomorphisms, where $X$ is a Riemann surface of genus $g$ and $\omega$ is a holomorphic differential whose zeros have orders prescribed by $\mu$. These strata generally fails to be connected and, in fact, for $g\ge2$ they generally exhibit up to three connected components according to the presence of a hyperelliptic involution or a topological invariant known as the spin structure, see \cite{KZ}.

\subsection{Absolute and relative periods}

\noindent In this section, we give a detailed description of the problem we are interested in. For a closed surface $S_{g}$, let $Z=\{z_o,\dots,z_k\}\subset S_{g}$ be a finite collection of $k+1$ points. The absolute homology $\textnormal{H}_1(\,S_{g},\,\mathbb Z\,)$ and the relative homology $\textnormal{H}_1(\,S_{g},\,Z,\,\mathbb Z\,)$ fit in a short exact sequence
\begin{equation}\label{eq:exactseq}
    0 \longrightarrow \textnormal{H}_1(\,S_{g},\,\mathbb Z\,) \longrightarrow \textnormal{H}_1(\,S_{g},\,Z,\,\mathbb Z\,) \longrightarrow \mathfrak{S}\longrightarrow 0
\end{equation}
\noindent where $\mathfrak S$ is the group generated by the relative periods, that is
\begin{equation}
    \mathfrak S \,=\, \left\langle \,\,\int_{z_o}^{z_i}\,\,\big|\,i=1,\dots,k\,\,\right\rangle,
\end{equation} 

\noindent and $z_o$ has been chosen as the base point. In general, there is no preferred path joining $z_o$ with another point in $Z$, and hence there is no preferred relative period. However, it is worth noticing that any two paths joining $z_o$ with $z_i$, for a fixed  $i=1,\dots,k$, always differ by a closed curve. As a consequence, two relative periods with the same starting point and endpoints always differ by an absolute period.  This leads us to state our problem as follows.

\smallskip

\begin{defn}
    Let $S_{g}$ be a closed surface and let $Z\subset S_{g}$ be a finite set. Let $\chi\colon \textnormal{H}_1(\,S_{g},\,Z,\,\mathbb Z\,)\longrightarrow \mathbb{C}$ be a representation and let \(\chi_a\) be the induced absolute period representation. Given a signature $\mu$ of holomorphic differentials, we say that $\chi$ \textit{can be realized} in $\mathcal H_g(\,\mu\,)$,  if there exists a translation surface $(X,\omega)\in\mathcal H_g(\,\mu\,)$ with absolute period character $\chi_a$ and, given its set of zeros $Z(\,\omega\,)$, then
    \begin{itemize}
        \item[1.] $|\,Z(\omega)\,|=|\,Z\,|$, and 
        \item[2.] there is a labeling $\sigma\colon \{\,0,\dots,k\,\}\longrightarrow Z(\,\omega\,)$ such that for any $i=1,\dots,k$ there is a path $\delta_i$ that joins the zeros $z_o$ and $z_i$ such that
    \begin{equation}
        \int_{\delta_i} \omega \,=\,\chi(\,\delta_i\,).
    \end{equation}
    \end{itemize}
\end{defn}

\noindent According to the definition above, in the present paper we aim to provide necessary and sufficient conditions under which a representation $\chi$ as in \eqref{eq:relrep} can appear as the period character of some element $(X,\omega)$ in a given stratum.

\begin{rmk}\label{rmk:notrirep}
    We remark that when $|\,Z\,|=1$, then $\chi\colon \textnormal{H}_1(S_{g\,},\,Z,\,\mathbb Z\,)\longrightarrow \mathbb{C}$ can be realized in the stratum $\mathcal H_g(\,2g-2\,)$ if and only if the absolute period character can be realized in the same stratum.
\end{rmk}

\medskip

\textit{Our convention.} In our setting, all relative periods are piecewise geodesics paths joining two distinct zeros. We allow relative periods to be self-intersecting, \textit{i.e.} passing through the same zero several times. However, notice that every path joining two distinct zeros has in its homotopy class a continuous embedded paths. We also allow two paths $\delta_i$ and $\delta_j$ to intersect at points other than $z_o$. Why do we need to allow these conditions? The reason is that, as already alluded above, in principle there is no canonical way to choose a relative period. Therefore, in our setting, we require that at least one path has the prescribed relative period. In particular, such a path may very well be the composition of some closed curves, not necessarily simple, based at $z_o$ with a given absolute period.

\subsection{Translation surfaces}\label{ssec:transurf}
As already alluded in \S\ref{ssec:gpovts}, a \textit{translation surface} is a branched $(\C,\C)$-structure, \textit{i.e.} the datum of a maximal atlas where local charts in $\C$ have the form $z\longmapsto z^k$, for $k\ge1$, with transition maps given by translations on their overlaps. Any atlas thus defined yields an underlying compact Riemann surface $X$ and the pullbacks of the $1$-form $dz$ on $\C$ via local charts globalize to a holomorphic differential, denoted by $\omega$ on $X$. Vice versa, a holomorphic differential $\omega$ on a compact Riemann surface $X$ defines a singular Euclidean metric with isolated singularities corresponding to the zeros of $\omega$. In a neighborhood of a point $\pto P$ which is not a zero of $\omega$, a local coordinate is defined as
\begin{equation}
    z(\,\pto Q\,)=\int_{\,\pto P}^{\,\pto Q} \omega 
\end{equation} 
in which $\omega=dz$, and the coordinates of two overlapping neighborhoods differ by a translation $z\mapsto z+c$ for some $c\in\mathbb C$. Around a zero, say $\pto P$ of order $k\ge1$, there exists a local coordinate $z$ such that $\omega=z^kdz$. The zero point $\pto P$ is also called a \textit{branch point} because any local chart around it is locally a branched covering of degree $k+1$ over $\C$ which is totally ramified at $\pto P$.

\begin{defn}[Translation surfaces
]\label{tswp}
Let $X\in\mathcal M_g$ be a compact Riemann surface and let $\omega\in\Omega(\,X\,)$ be a holomorphic differential. We define the corresponding \textit{translation surface} to be the geometric structure on $X$ induced by $\omega$. 
\end{defn}

\noindent 
Given a translation structure $(X,\omega)$ on a closed surface $S_{g}$ with \(g\ge2\), local charts globalize to a multivalued function, say \(F\colon X \longrightarrow \mathbb C\), that lifts to a holomorphic \(\rho-\)equivariant map \(\textnormal{dev}\colon \mathbb H\longrightarrow\mathbb C\), where \(\mathbb H\) is the universal cover of \(X\) and \(\rho\colon\pi_1(\,X\,)\longrightarrow \mathbb C\) is a representation. Note that \(\rho\) naturally reduces to a representation in homology \(\chia\colon\ahomolz\longrightarrow \mathbb C\) because the target is abelian. Vice versa, a \(\rho-\)equivariant developing map \(\textnormal{dev}\colon \mathbb H\longrightarrow\mathbb C\), determines a branched \((\,\mathbb C,\,\mathbb C\,)\)-structure, namely a translation surface \((X,\omega)\), by using local inverses of the covering projection \(\mathbb H\longrightarrow X\). Notice that here we have tacitly fixed a conformal identification \(\widetilde X\cong \mathbb H\) where \(\pi_1(\,X\,)\) acts on \(\mathbb H\) by M\"obius transformations. It is easy to verify that 
\begin{equation}\label{eq:perchar}
    \chia(\,\gamma\,)=\int_\gamma\omega,
\end{equation}
and hence it agrees with the period character defined by \(\omega\) on \(X\). Thus a period character is nothing but the \textit{holonomy representation} of the translation structure defined by $\omega$. The following Lemma establishes the relation between holonomy representations and period characters.

\begin{lem}
A representation $\chia\colon\ahomolz\longrightarrow \C$ is the period of some abelian differential $\omega\in\Omega(\,X\,)$ with respect to some complex structure $X$ on $S_{g}$ if and only if it is the holonomy of the translation structure on $S_{g}$ determined by $\omega$.
\end{lem}

\noindent This twofold nature of a representation $\chia$ is what permits to tackle our problem by adopting a geometric approach as explained in \S\ref{ssec:gpovts}. More precisely, in order to realize a representation $\chia$ on a given stratum of differentials, we will  realize it as the holonomy of some translation surface, say $(X,\omega)$, with prescribed zeros and, whenever they are defined, with prescribed spin structure or hyperelliptic involution. 

\begin{rmk}\label{rmk:STS}
    There is a special class of translation surfaces $(X,\omega)$ arising as finite covers $\pi\colon X\longrightarrow \mathbb C\,/\,\ziz$ branched at only one point and $\omega=\pi^*dz$. These are  well-known in literature as \textit{origami} and they form a class of interesting and rich objects to study. Notice that these surfaces are naturally tiled by unit squares, and hence they are commonly known also as \textit{square-tiled surfaces}. For our purposes in \S\ref{sec:sqrspesper}, we need to extend this notion to a broader class of translation surfaces by allowing more branch points. Since all surfaces in this larger class are also square-tiled, in what follows we convey that a square-tiled surface is simply a translation surface arising from a finite branched cover over $\mathbb C\,/\,\ziz$ and reserve origami for those covers that are branched at a single point. See also Definition \ref{defn:transurf}.
\end{rmk}

\begin{rmk}\label{partrans}
Let $(X,\omega)$ be a translation surface, possibly with poles, let $X^*$ denote $X\setminus Z(\,\omega\,)$ and pick any point $x_o\in X^*$. Since the structure is flat, the parallel transport induced by the flat connection yields a homomorphism from $\pi_1(X^*,\,x_o)$ to $\textnormal{SO}(2,\mathbb R)\cong\mathbb S^1$ which acts on the tangent space of $x_o$. Since $\mathbb S^1$ is abelian, such a homomorphism factors through the homology group, and hence we have a well-defined homomorphism $\text{PT}\colon\textnormal{H}_1(X^*,\,\mathbb Z)\to \textnormal{SO}(2,\mathbb R)\cong\mathbb S^1$. For translation surfaces, this homomorphism turns out to be trivial in the sense that a parallel transport of a vector tangent to the Riemann surface $X$ along any closed path avoiding the zeros of $\omega$ brings the vector back to itself.
\end{rmk}

\begin{rmk}\label{rmk:foliations}
Let $x_o\in(X,\omega)$ be a regular point, \textit{i.e.} $x_o$ is not a zero of $\omega$. A given tangent direction at $x_o$ can be extended to all other regular points by means of the parallel transport. This yields a non-singular foliation which extends to a singular foliation with singularities at the branch points. Let $z=x+iy$ be a local coordinate at $x_o$. The \textit{horizontal foliation} is the oriented foliation determined by the positive real direction in the coordinate $z$. Notice that this is well-defined because different local coordinates differ by a translation. In the same fashion, the \textit{vertical foliation} is the oriented foliation determined by the positive purely imaginary direction in the coordinate $z$.
\end{rmk}

\subsection{Volume}\label{sssec:vol} We now discuss the notion of \textit{volume} which plays an important role in the theory. For our purposes, we are mostly interested in the algebraic volume which is a topological invariant associated to a representation $\chi$.
\smallskip

\noindent Let us recall this notion in the holomorphic case. For a symplectic basis $\{\alpha_1,\beta_1,\dots,\alpha_g,\beta_g\}$ of $\shomolz$, we define the volume of a representation $\chi\colon\shomolz\longrightarrow \C$ as the quantity
\begin{equation}\label{algvol}
    \text{vol}(\,\chi\,)=\text{vol}(\,\chia\,)=\sum_{i=1}^g \Im\Big(\,\,\overline{\chi(\alpha_i)}\,\chi(\beta_i)\,\,\Big),
\end{equation} 
where $\Im(\,\overline{z}\,w\,)$ is the usual symplectic form on $\C$. As a consequence, this algebraic definition of volume of a character $\chi$ is invariant under precomposition with any automorphisms in $\textnormal{Aut}\,\big(\shomolz\big)\cong\spz$. The image of $\chi$, provided that it has rank $2g$, turns out to be a polarized module.

\begin{rmk}\label{rmk:capersp} We provide here an alternative definition of volume by adopting a complex-analytic perspective. Let $X$ be a compact Riemann surface and let $\omega\in\Omega(\,X\,)$ be a holomorphic differential with period character $\chia$. Let \(F\colon X\longrightarrow \mathbb C\) be a solution of the differential equation \(dF=\omega\) with monodromy representation \(\chia\) as in \S\ref{ssec:gpovts}. Then, we can define the volume of \(\chia\) as
\begin{equation}
\frac{\rm i}{2}\int_X F^*\,\Big( \,dz\wedge d\overline{z}\,\Big)=\frac{\rm i}{2}\int_X \omega\wedge\varpi\,, 
\end{equation}
that is the area of the singular Euclidean metric determined by $\omega$ on $X$. In particular, by means of Riemann's bilinear relations one can show that it agrees with the algebraic definition of volume as in \eqref{algvol}. The reader can consult \cite{NR}.
\end{rmk} 

\smallskip

\textit{Haupt's conditions.} As already alluded in the introduction, there are some obstructions for realizing $\chi$ as the period character of some holomorphic differential $\omega$ on a compact Riemann surface $X$, which are caused by the volume. We recall them here for the reader's convenience. The first requirement is that the volume of $\chi$ has to be positive with respect to some symplectic basis of $\shomolz$. Indeed, the volume equals the area of the surface $X$ endowed with the singular Euclidean metric induced by $\omega$, see Remark \ref{rmk:capersp}. There is a second obstruction that applies in the case $g\geq 2$ and only when the image of $\chi\colon\shomolz\longrightarrow\C$ is a lattice, say $\Lambda$ in $\C$. In fact, it is possible to show that in this special case, $(X,\omega)$ arises from a branched cover of the flat torus $\C/\Lambda$. In particular, the following inequality must hold: 
\begin{equation}\label{eq:hcond2}
    \textnormal{vol}(\,\chia\,)\ge 2\,\text{{area}}(\mathbb C/\Lambda)
\end{equation}
\noindent Haupt's Theorem states that these are the only obstructions for realizing $\chi$ as the period character of some holomorphic differential. 

\smallskip

\noindent These conditions provide necessary and sufficient criteria for the existence of a translation surface with a prescribed period character. However, they do not give any condition for realizing such a surface within a specific connected component of a stratum. In their works, \cite{fils} and \cite{BJJP} independently provided necessary conditions for the realization of a given representation within a specific stratum. If the image of a given representation is \textit{non-discrete}, then the only obstruction to its realization in any connected component of any stratum is the sign of the volume, which must be positive, \textit{i.e.} \(\textnormal{vol}(\,\chia\,)>0\). However, if the image of the representation is a lattice, condition \eqref{eq:hcond2} is no longer sufficient, and the realization within a connected component of a stratum of holomorphic differentials, say \( \mathcal{H}_g(a_0, \dots, a_k) \), is subject to a strengthened version of condition \eqref{eq:hcond2} given as follows:
\begin{equation}\label{eq:hcond2s}
    {\rm vol}(\,\chia\,) \ge \max_{0\le i\le k} \big(\,a_i+1\,\big)\,{\rm area}(\,\mathbb C/\Lambda\,).
\end{equation}
In order to realize a relative period representation \(\chi\), in \S\ref{sec:sqrspesper} we need to strength condition \eqref{eq:hcond2s} even further; see conditions \eqref{eq:neccondtoreal}, \eqref{eq:hol-necessary}, and Remark \ref{rmk:shcd}.

\smallskip

\subsection{Further invariants}\label{ssec:invariants} We now briefly recall a few invariants naturally attached to an absolute period representation $\chia$. We will use two of these to distinguish the connected components of strata whenever they are disconnected. For a more detailed treatment of these invariants in the holomorphic setting, the reader can consult \cite{KZ}.

\smallskip

\subsubsection{Hyperelliptic involution} As already alluded in \S\ref{ssec:strata}, a stratum can be disconnected and its connected components are distinguished by two types of invariants. One of these is \textit{hyperellipticity} or \textit{hyperelliptic structure}. In what follows, we make a direct use of this notion only when $\chi$ has discrete absolute period character $\chia$. We begin with the following

\smallskip

\begin{defn}\label{hypdef}
A translation surface $(X,\omega)$ is said to be \textit{hyperelliptic} if $X$ is hyperelliptic and $\omega$ is anti-invariant under the hyperelliptic involution $\tau$, \textit{i.e.} $\tau^*\omega=-\omega$.
\end{defn}

\noindent For a hyperelliptic translation surface $(X, \omega)$, the hyperelliptic involution $\tau$ realizes an isometry between the singular Euclidean structures determined by the differentials $\omega$ and $-\omega$ on $X$. Therefore, not all strata can admit hyperelliptic translation surfaces. In fact, if $\omega$ has two zeros that are swapped by $\tau$, then they must have the same order. Additionally, if a zero is fixed by $\tau$, then it must have even order.  

\smallskip

\subsubsection{Spin parity}\label{ssec:spin} Another geometric invariant we need to introduce is the spin structure. Let $(X,\omega)$ be a translation surface. Recall from Remark \ref{rmk:foliations} that away from $Z(\,\omega\,)$ there is a well-defined horizontal direction and hence a non-singular horizontal foliation on $X\setminus Z(\,\omega\,)$. Such a foliation extends to a singular horizontal foliation over the zeros of $\omega$. In fact, for a zero, say $\pto P$ of order $k$, there are exactly $k+1$ horizontal directions leaving from $\pto P$. Let $\gamma$ be a simple smooth curve parametrized by the arc length. Assume in addition that $\gamma$ avoids all the zeros of $\omega$. Then, we define the index of $\gamma$ as a numerical invariant given by the comparison of the unit tangent field $\dot\gamma(\,t\,)$ and the unit vector field along $\gamma$ given by unit vectors tangent to the horizontal direction. More precisely, denote by $u(\,t\,)$ the unit vector at $\gamma(\,t\,)$ tangent to the leaf of the horizontal foliation through $\gamma(\,t\,)$. Then, the assignment 
\begin{equation}\label{comphorfol}
    t\longmapsto \theta(\,t\,)=\angle\,\Big(\,\dot\gamma(\,t\,),\,u(\,t\,)\,\Big)
\end{equation} defines a mapping $f_\gamma\colon\mathbb S^1\longrightarrow \mathbb S^1$.

\begin{defn}\label{index}
The index of $\gamma$ is defined as $\deg(\,f_\gamma\,)$ and denoted by $\textnormal{Ind}(\,\gamma\,)$. 
\end{defn}

\smallskip

\noindent The index $\textnormal{Ind}(\,\gamma\,)$ of a closed curve $\gamma$ measures the number of times the unit tangent vector field spins with respect to the direction given by the horizontal foliation. Since for translation surfaces the parallel transport yields a trivial homomorphism, see Remark \ref{partrans}, the total change of the angle is $2\pi\,\textnormal{Ind}(\gamma)$.

\smallskip

\textit{Convention.} We use the convention that $\mathbb S^1$ is counterclockwise oriented. As a consequence, the index $\textnormal{Ind}(\,\gamma\,)=\deg(\,f_\gamma\,)$ is \textit{positive} if  $\gamma$ is counterclockwise oriented.

\medskip

\noindent Following the earlier work of Johnson \cite{JO} (see also \cite{AM} and \cite{MD}), we can define the \textit{spin structure} as the parity of 
\begin{equation}\label{eq:spinparity}
    \varphi(\,\omega\,)=\sum_{i=1}^g \big(\,\text{Ind}(\alpha_i)+1\big)\big(\,\text{Ind}(\beta_i)+1\big) \,\,(\text{mod}\,2).
\end{equation}

\noindent determined by an abelian differential $\omega$, where $\{\alpha_1,\beta_1,\dots,\alpha_g,\beta_g\}$ is a symplectic basis of $\textnormal{H}_1(X,\, \mathbb Z)$. Notice that the spin parity is only well-defined when all the singularities have even orders, because \(\gamma\) passing through a singularity of order \(m\) changes its index by \(m\). Thus the spin structure is a topological invariant well-defined for translation surfaces \((X,\omega)\) of \textit{even type}, meaning that all zeros of \(\omega\) have even order. For more details, see \cite{CF} and \cite{KZ}. It is straightforward to check that $\varphi(\,\omega\,)$ does not depend on the choice of the symplectic basis. Finally, it follows from the aforementioned references that the spin parity is invariant under continuous deformations.

\medskip

\section{Surgeries on translation surfaces}\label{sec:surgeries} 

\noindent In this section, we recall a few surgeries we will use in the sequel. These surgeries all consist of cutting a given translation surface along one, or possibly more, geodesic segment(s) and gluing them back along those segments in order to get new translation structures. Different ways of gluing will provide different translation structures. In order to introduce these operations, we need to make use of the following terminology already adopted in \cite{CF}.

\smallskip

\subsection{Convention and terminology}\label{sssec:conv} Recall that slitting a surface along an oriented geodesic segment $s$ is a topological surgery for which the interior of $s$ is replaced with two copies of itself. On the resulting surface, these two segments, form a piecewise geodesic boundary with two corner points, corresponding to the extremal points of $s$, each of which can be a regular point or a branch point of angle $2(m+1)\pi$ for some integer $m\ge1$. We denote by $s^+$ the piece of boundary which bounds the surface on its right with respect to the orientation induced by $s$. In a similar fashion, we denote by $s^-$ the piece of boundary which bounds the surface on its left with respect to the orientation induced by $s$. See \textit{e.g.}~Figure~\ref{slit} for the $\pm$ labeling convention. 

\begin{rmk}\label{rmk:conv}
    In the following we need to extend this convention to sides of polygons. More precisely, for a polygon $\mathcal P$ (possibly non-compact), we denote sides that are identified via a translation with the same label up to the sign. Once the orientation is fixed, we will use the sign $+$ for the side that leaves the polygon on its right and the sign $-$ for the other side, namely the one that leaves the polygon on its left.
\end{rmk}

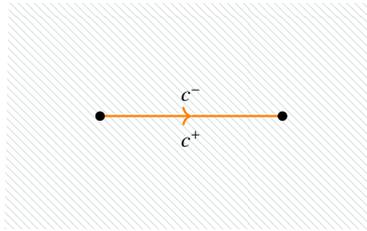
\begin{figure}[ht!]
    \centering
    \begin{tikzpicture}[scale=1.2, every node/.style={scale=0.8}]
    \definecolor{pallido}{RGB}{221,227,227}
    
    \pattern [pattern=north west lines, pattern color=pallido]
    (-4.5,-0.25)--(-0.5,-0.25)--(-0.5,2.25)--(-4.5,2.25)--(-4.55,-0.25);
    \draw [thick, orange, ->] (-3.5,1) to (-2.5,1);
    \draw [thick, orange] (-2.5,1) to (-1.5,1);
    \fill (-3.5,1) circle (1.5pt);
    \fill (-1.5,1) circle (1.5pt);

    \node at (-2.5, 1.25) {$c^-$};
    \node at (-2.5, 0.75) {$c^+$};
   
    \end{tikzpicture}
    \caption{Labels of a slit}
    \label{slit}
\end{figure}

\noindent Sometimes we will omit the arrows as the direction is implicitly understood by the signs. In the following constructions, we need to slit and glue surfaces along geodesic segments. In order for this operation to be done, we need to glue along segments which are parallel after developing, \textit{e.g.} see \S\ref{sssec:hdhyp}. For any $c\in\C^*$, by \textit{slitting $(\C,\,dz)$ along $c$, we mean a cut along any geodesic segment $s$ of length $|\,c\,|$ and direction equal to $\arg(c)$, that is 
\begin{equation}
    \int_s dz=c.
\end{equation}} 

\medskip

\subsection{Breaking up a zero}\label{sec:zerobreak} The surgery we are going to describe has been introduced by Eskin--Masur--Zorich in \cite[Section 8.1]{EMZ} and it literally "breaks up" a zero in two, or possibly more, zeros of lower orders. Complex-analytically, this can be thought as the analogue to the classical Schiffer variations for Riemann surfaces, see \S\ref{sssec:schiffer} and \cite{NS} for more details. 
This surgery only modifies a translation surface on a contractible neighborhood of the initial zero. In particular, after the surgery the resulting surface has the same genus as the former one, but the type of zero orders is changed. Moreover, the new translation surface we obtain after the surgery has the same (absolute) period character as the original one. As a consequence, this operation produces small deformations of the original translation structure in the same isoperiodic fiber.

\begin{rmk} In the context of branched projective structures, such a surgery is also known as the \textit{movement of branched points} and it has been originally introduced by Tan in \cite[Chapter 6]{TA} for showing the existence of a complex one-dimensional continuous family of infinitesimal deformations of a given structure.\end{rmk}

\noindent Let us now explain this surgery in more detail. Let $(X,\omega)$ be a translation surface possibly with poles. Breaking up a zero is a procedure that takes place at the $\varepsilon$-neighborhood of some zero of order $m$ of the differential on which it looks like the pull-back of the form $dz$ via a branched covering $z\longmapsto z^{\,a+1}$. The differential is then modified by a surgery inside this $\varepsilon$-neighborhood. Once this surgery is performed we obtain a new translation structure with two zeros of order $a_1$ and $a_2$ such that $a_1+a_2 = a$. Furthermore, the translation structure remains unchanged outside the $\varepsilon$-neighborhood of such a zero of order $a$. The idea is to consider the $\varepsilon$-neighborhood of a zero of order $a$ as $a+1$ copies of a disc $D$ of radius $\varepsilon$ whose diameters are identified in a specified way. We can see this family of discs as a collection of $a+1$ upper half-discs and $a+1$ lower half-discs. See figure \ref{fig:zerolocmodel}.

\smallskip

\begin{figure}[ht]
    \centering
    \begin{tikzpicture}[scale=0.7, every node/.style={scale=0.6}]
    \definecolor{pallido}{RGB}{221,227,227}
    \foreach \x [evaluate=\x as \coord using 4 + 5*\x] in {0, 1, 2} 
    {
    \draw [pattern=north west lines, pattern color=pallido] (\coord,0) arc [start angle = 0,end angle = 180,radius = 2];
    \draw [pattern=north west lines, pattern color=pallido] (\coord,-1) arc [start angle = 0,end angle = -180,radius = 2];
    }
    \foreach \x [evaluate=\x as \leftend using 5*\x] [evaluate=\x as \rightend using 2 + 5*\x] in {0,1,2} 
    {
    \draw [thick] (\leftend, 0) -- (\rightend, 0);
    \draw [thick] (\leftend, -1) -- (\rightend, -1);
    }
    \foreach \x [evaluate=\x as \leftlabel using 1 + 5*\x] [evaluate=\x as \rightlabel using 3 + 5*\x] in {0, 1, 2} 
    {
    \node [below] at (\leftlabel, 0) {$u_{r\x}$};
    \node [below] at (\rightlabel, 0) {$u_{l\x}$};
    \node [above] at (\leftlabel, -1) {$l_{l\x}$};
    \node [above] at (\rightlabel, -1) {$l_{r\x}$};
    }
    
\foreach \botindex / \topindex / \colr [evaluate=\botindex as \botleftend using 5*\botindex + 2] [evaluate=\botindex as \botrightend using 5*\botindex + 4] [evaluate=\topindex as \topleftend using 5*\topindex + 4] [evaluate=\topindex as \toprightend using 5*\topindex + 2] in {0/1/blue, 1/2/red, 2/0/purple} 
    {
    \draw [thick, \colr] (\topleftend, 0) -- (\toprightend, 0);
    \draw [thick, \colr] (\botleftend, -1) -- (\botrightend, -1);
    \fill (\toprightend, 0) circle (1.5pt);
    \fill (\botleftend, -1) circle (1.5pt);
    }
 
    \end{tikzpicture}
    \caption{An $\varepsilon$-neighborhood of a zero of order $2$.}
    \label{fig:zerolocmodel}
\end{figure}
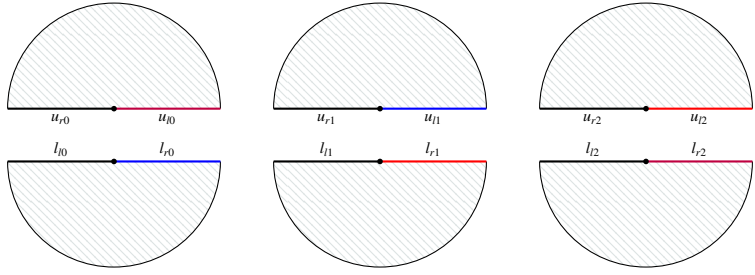

\noindent  We now break up a zero of order $a$ into two zeros of order $a_1$ and $a_2$. This surgery consists of identifying the diameters of the starting $a+1$ discs in a different way as follows. First, we modify the labeling on the upper half disc indexed by $0$, the lower half disc indexed by $a_1$, and all upper and lower half discs with index more than $a_1$ accordingly. The modified labeling is shown below in Figure \ref{fig:splitlocmodel} for the case of splitting a zero of order $2$ into two zeros of order $1$.

\smallskip

\begin{figure}[ht]
    \centering
    \begin{tikzpicture}[scale=0.7, every node/.style={scale=0.6}]
    \definecolor{pallido}{RGB}{221,227,227}
    \foreach \x [evaluate=\x as \coord using 4 + 5*\x] in {0, 1, 2} 
    {
    \draw [pattern=north west lines, pattern color=pallido] (\coord,0) arc [start angle = 0,end angle = 180,radius = 2];
    \draw [pattern=north west lines, pattern color=pallido] (\coord,-1) arc [start angle = 0,end angle = -180,radius = 2];
    }
    \foreach \x [evaluate=\x as \leftend using 5*\x] [evaluate=\x as \rightend using 2 + 5*\x] in {0, 1, 2} 
    {
    \draw [thick] (\leftend, 0) -- (\rightend, 0);
    \draw [thick] (\leftend, -1) -- (\rightend, -1);
    \node[above] at (\rightend, 0) {$P$};
    \node[below] at (\rightend, -1) {$P$};
    }
    \draw [thick, orange] (2,0) -- (2.8,0);
    \fill (2,0) circle (1.5pt);
    \draw [thick, orange] (12,-1) -- (12.8,-1);
    \fill (12,-1) circle (1.5pt);
    \node [below] at (2.4, 0) {$um$};
    \node [above] at (12.4, -1) {$lm$};
    \foreach \botindex / \topindex / \colr [evaluate=\botindex as \botleftend using 5*\botindex + 2] [evaluate=\botindex as \botrightend using 5*\botindex + 4] [evaluate=\topindex as \topleftend using 5*\topindex + 2] [evaluate=\topindex as \toprightend using 5*\topindex + 4] [evaluate=\topindex as \topsplpt using 5*\topindex + 2.8] [evaluate=\botindex as \botsplpt using 5*\botindex + 2.8] in {0/1/blue, 1/2/red} 
    {
    \draw [thick, \colr] (\topleftend, 0) -- (\toprightend, 0);
    \draw [thick, \colr] (\botleftend, -1) -- (\botrightend, -1);
    \fill (\topleftend, 0) circle (1.5pt);
    \fill (\botleftend, -1) circle (1.5pt);
    \fill (\topsplpt, 0) circle (1.5pt);
    \fill (\botsplpt, -1) circle (1.5pt);
    \node[above] at (\topsplpt, 0) {$Q$};
    \node[below] at (\botsplpt, -1) {$Q$}; 
    }
    \foreach \botindex / \topindex / \colr [evaluate=\botindex as \botleftend using 5*\botindex + 2.8] [evaluate=\botindex as \botrightend using 5*\botindex + 4] [evaluate=\topindex as \topleftend using 5*\topindex + 2.8] [evaluate=\topindex as \toprightend using 5*\topindex + 4] in {2/0/violet} 
    {
    \draw [thick, \colr] (\topleftend, 0) -- (\toprightend, 0);
    \draw [thick, \colr] (\botleftend, -1) -- (\botrightend, -1);
    \fill (\topleftend, 0) circle (1.5pt);
    \fill (\botleftend, -1) circle (1.5pt);
    \node[above] at (\topleftend, 0) {$Q$};
    \node[below] at (\botleftend, -1) {$Q$};
    }
    \foreach \x [evaluate=\x as \leftlabel using 1 + 5*\x] in {0, 1, 2} 
    {
    \node [below] at (\leftlabel, 0) {$ul_{\x}$};
    \node [above] at (\leftlabel, -1) {$ll_{\x}$};
    }
    \foreach \botindex / \topindex [evaluate=\botindex as \botlabel using 3+ 5*\botindex] [evaluate=\topindex as \toplabel using 3+ 5*\topindex] in {0/1, 1/2} 
    {
    \node [below] at (\toplabel, 0) {$ur_{\topindex}$};
    \node [above] at (\botlabel, -1) {$lr_{\botindex}$};
    }
    \foreach \botindex / \topindex [evaluate=\botindex as \botlabel using 3.4+ 5*\botindex] [evaluate=\topindex as \toplabel using 3.4+ 5*\topindex] in {2/0} 
    {
    \node [below] at (\toplabel, 0) {$ur_{\topindex}$};
    \node [above] at (\botlabel, -1) {$lr_{\botindex}$};
    }
    \end{tikzpicture}
    \caption{New labeling for breaking up a zero of order $2$ in two zeros of order $1$.}
    \label{fig:splitlocmodel}
\end{figure}
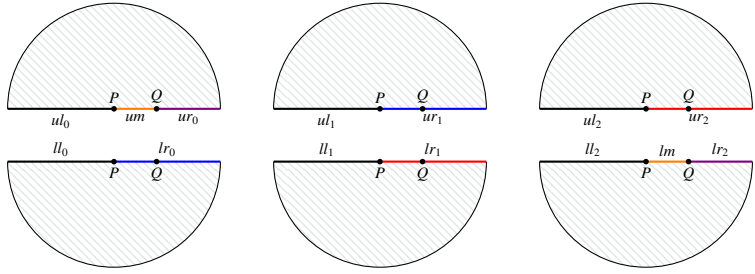

\noindent Next, we identify $ul_i$ with $ll_i$ and $lr_i$ with $ur_{i+1}$ as before with the added identification of $um$ with $lm$. This identification gives two zeros $A$ and $B$, where $A$ is a zero of the differential of order $m_1$ and $B$ is a zero of order $m_2$. We also get a geodesic line segment joining $A$ and $B$. Given $c \in \mathbb{C}\setminus\{0\}$ with length less than $2\varepsilon$, we can perform the surgery in such a way that the line segment joining $A$ and $B$ is $c$. It is also clear that such a deformation of the translation structure is only local. This procedure can be repeated multiple times to obtain zeros of orders $a_1, \dots, a_k$ from a single zero of order $a_1 + \cdots + a_k$. We will frequently rely on this procedure of breaking up a zero in the remaining part of this paper. 

\smallskip

\subsection{Schiffer variations and movement of zeros}\label{sssec:schiffer} As already alluded above, Schiffer variations are intimately related to the surgery just described in Paragraph \S\ref{sec:zerobreak}. Geometrically, Schiffer variations can be seen as a \textit{movement of zeros} along a given direction. 

\begin{rmk}
    By adopting the language of geometric structures on surfaces, the movement of zeros is largely known as the \textit{movement of branch points}, see \cite{CDF}, \cite{CDF2}, and \cite{TA} where, in that case, a \textit{branch point} is understood as a point around which local charts are branch covers onto some open subsets of the model space. In this language, translation surfaces are branched geometric structures, and the zeros thus correspond to the branch points. In the present paper, however, we do not adopt this terminology.
\end{rmk}

\noindent Before stating the movement of zeros, let us introduce the following terminology as in \cite[Section \S8]{CF}. We remark that this is a different but equivalent formulation compared to the other currently known in literature, see \cite[Section \S2]{CDF}.

\begin{defn}[Twin paths]\label{def:twins}
On a translation surface $(X,\omega)$, let $P$ be a branch point of order $m$. Consider a collection of $m+1$ embedded path
$c_i\colon [\,0,\,1\,]\longrightarrow (X,\omega)$ such that $c_i(\,0\,)=P$ for $i=0,\dots,m$, each of which is injectively developed, and all of which overlap once developed, \textit{i.e.} there is a determination of the developing map around $c_0\,\cup\,\cdots\,\cup\,c_m$ such that
\begin{itemize}
    \item $c_0,\,c_1,\dots,c_m$ injectively develop to the same arc $\widehat{c}\subset \C$, and 
    \item $\text{dev}\big(\,c_i(\,t\,)\,\big)=\text{dev}\big(\,c_j(\,t\,)\,\big)$ for every $t\in[0,1]$ and $i,j\in\{0,\dots,m\}$.
\end{itemize}
\noindent For any $2\le k\le m$, the paths $c_{i_1},\dots,c_{i_k}$ are called \textit{twin paths}. For any pair $c_i,\,c_j$, notice that the angle at $P$ between them is a multiple of $2\pi$.
\end{defn}

\textit{Convention.} In what follows we assume that all twin paths $c_0,\dots,c_m$ at $P$ as defined in the above are counterclockwise ordered if not otherwise specified.

\smallskip

\begin{defn}[Movement of zeros]\label{def:move} Let $(X,\omega)$ be a translation surface and let $P$ be a zero of order $m\ge1$. Let $c_o,\dots,c_m$ be a family of embedded paths at $P$ such that for any $i,j$, the pair $c_i,\,c_j$ turns out to be a couple of embedded twin paths and $c_i(1)$ is a regular point for any $i=0,\dots,m$. Cut $(X,\omega)$ along $c_o\cup\dots\cup c_m$ and gluing back $c_i^+$ with $c_{i+1}^-$ for $i=0,\dots,m-1$ and also $c_o^-$ with $c_m^+$. The resulting structure is a local deformation of $(X,\omega)$ obtained by a \textit{movement of zeros}. In what follows, whenever it is necessary, we will denote
by ${\rm Move}\big(\,(X,\omega),\,{\textbf{c}}\,\big)$ such a deformation of $(X,\omega)$, where ${\textbf{c}}$ denotes a collection $\left\{\,c_o,\dots,c_m\,\right\}$ of twin paths.
\end{defn}
 
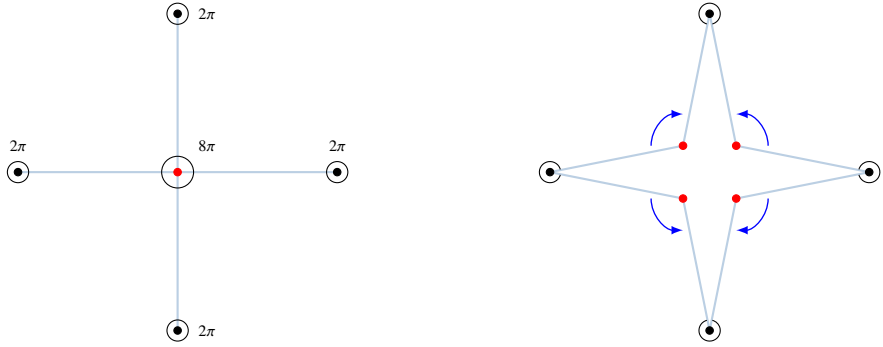
\begin{figure}[!h]
\centering
\begin{tikzpicture}[scale=0.7, every node/.style={scale=0.6}]

\definecolor{pallido}{RGB}{187,207,227}

\draw [thick, pallido] (0,5) to (3,5);
\draw [thick, pallido] (3,5) to (6,5);
\draw [thick, pallido] (3,5) to (3,8);
\draw [thick, pallido] (3,5) to (3,2);

\draw [thick, pallido] (12.5,5.5) to (13,8);
\draw [thick, pallido] (12.5,4.5) to (13,2);
\draw [thick, pallido] (13.5,5.5) to (13,8);
\draw [thick, pallido] (13.5,4.5) to (13,2);

\draw [thick, pallido] (10,5) to (12.5,5.5);
\draw [thick, pallido] (10,5) to (12.5,4.5);
\draw [thick, pallido] (16,5) to (13.5,5.5);
\draw [thick, pallido] (16,5) to (13.5,4.5);

 \draw (16,5) ++ (-165:0.2) arc (-165:165:0.2);
 \draw (10,5) ++ (15:0.2)  arc (15:345:0.2);
 \draw (13,8) ++ (-75:0.2)  arc (-75:255:0.2);
 \draw (13,2) ++ (-255:0.2)  arc (-255:75:0.2);

\draw [red]    plot [mark=*, smooth] coordinates {(3,5)};
\draw [black]  plot [mark=*, smooth] coordinates {(0,5)};
\draw [black]  plot [mark=*, smooth] coordinates {(6,5)};
\draw [black]  plot [mark=*, smooth] coordinates {(3,8)};
\draw [black]  plot [mark=*, smooth] coordinates {(3,2)};

\draw [thin] (3,5) circle (3mm);
\draw [thin] (0,5) circle (2mm);
\draw [thin] (6,5) circle (2mm);
\draw [thin] (3,8) circle (2mm);
\draw [thin] (3,2) circle (2mm);

\draw [red] plot [mark=*, smooth] coordinates {(12.5,5.5)};
\draw [red] plot [mark=*, smooth] coordinates {(12.5,4.5)};
\draw [red] plot [mark=*, smooth] coordinates {(13.5,5.5)};
\draw [red] plot [mark=*, smooth] coordinates {(13.5,4.5)};

\draw [black] plot [mark=*, smooth] coordinates {(10,5)};
\draw [black] plot [mark=*, smooth] coordinates {(16,5)};
\draw [black] plot [mark=*, smooth] coordinates {(13,8)};
\draw [black] plot [mark=*, smooth] coordinates {(13,2)};

\node at (3.55,5.5) {$8\pi$};
\node at (3.55,8) {$2\pi$};
\node at (3.55,2) {$2\pi$};
\node at (6,5.5) {$2\pi$};
\node at (0,5.5) {$2\pi$};

\draw [latex- ,line width=0.5pt, blue] (12.5, 5.5) ++(90:6mm) arc (90:180:6mm);
\draw [-latex ,line width=0.5pt, blue] (12.5, 4.5) ++(180:6mm) arc (180:270:6mm);
\draw [latex- ,line width=0.5pt, blue] (13.5, 5.5) ++(90:6mm) arc (90:0:6mm);
\draw [-latex ,line width=0.5pt, blue] (13.5, 4.5) ++(0:6mm) arc (0:-90:6mm);
\end{tikzpicture}
\caption[]{Movement of a zero of order $3$.}\label{sm}
\end{figure}

\noindent In the study of branched structures on surfaces, the importance of such a surgery is its key property of preserving the holonomy of the structure. Moreover, it does not change the structure of the branching divisor. In our language, the holonomy being preserved means that the absolute periods remain untouched once the surgery is performed. In particular, the resulting structure lies in the same stratum.

\medskip

\indent \textit{Geometric interpretation of "movement".} We would like to provide an alternative description of the movement of zeros by using a more geometric and dynamical interpretation of the "movement". 

\smallskip

\noindent For this purpose, let $(X,\omega)$ be any translation surface, possibly with poles. Let $P$ be a zero of $\omega$ of order $m\ge1$ and let $B=B_\varepsilon(P)$ be an open metric ball embedded in $(X,\omega)$. For $\varepsilon$ small enough, $B$ homeomorphically lifts to an open ball in $(\widetilde X, \widetilde\omega)$ and we can assume that the restriction of the developing map is a local branched chart $\varphi\colon B\longrightarrow \mathbb C$ of degree $m+1$,  branched at $P$. Let $\widehat P=\varphi(\,P\,)$ and let $\widehat B=\varphi(\,B\,)$. Let $\textbf{c}=\left\{\,c_o,\dots,c_m\,\right\}$ be a collection of twin paths based at $P$ and let $Q_i$ denote the extreme point of $c_i$ other than $P$. Assume that $\textbf{c}$ is entirely contained in $B$. Since all paths in \textbf{c} are assumed to be twins, they all develop onto the same segment, say $\widehat c\subset \mathbb C$, having $\widehat P$ and $\widehat Q$ as the extremal points. Notice that $\varphi(\,Q_i\,)=\widehat Q$ for every $i=0,\dots,m$. In this setting, since $\varphi$ is a local branched cover, then $\widehat P$ is a branch value and, by assuming all $Q_i$ to be regular points (hence they are not zeros nor poles for $\omega)$, $\widehat Q$ is a regular value.

\smallskip

\noindent We then perform the surgery described in Definition \ref{def:move}. After the cut and paste process, $(X,\omega)$ is thus deformed to a new structure ${\rm Move}\big(\,(X,\omega),\,{\textbf{c}}\,\big)$. In this latter, the local chart $\varphi$, seen as a branched cover, has been deformed to a new local branched chart $\psi\colon B\longrightarrow \widehat B\subset \mathbb C$ of degree $m+1$ in which $\widehat P$ is a regular value whereas $\widehat Q$ is a branch value. More precisely, $\psi^{-1}(\,\widehat Q\,)$ has only one preimage, say $Q$, arising from the collapsing of $\{Q_o,\dots,Q_m\}$ and  $\psi^{-1}(\,\widehat P\,)$ has $m+1$ preimages $\{P_o,\dots, P_m\}$. The mapping $\psi$ around every $P_i$ is now a regular chart.

\smallskip

\noindent By keeping in mind this description, we can say that a zero of $\omega$ has been moved to another point and the movement is seen in local charts. For a more dynamical interpretation, we observe that the movement of zeros naturally comes with a $1-$parameter family. In the above notation, for every  $t\in[0,\,1]$ we can consider the structure ${\rm Move}\big(\,(X,\omega),\,{\textbf{c}^t}\,\big)$, where $\textbf{c}^t$ is the collection of embedded twin paths made of those sub-arcs of $c_i$'s parametrized by the sub-interval $[0,\,t]$. For $t$ running in the unit interval $[0,\,1]$, we see $\widehat P$, \textit{i.e.} the developed image of $P$, \textit{moving} along $\widehat c$,  which motivates the name of this surgery.

\begin{rmk}
    Breaking up a zero, seen as the operation just described in Paragraph \S\ref{sec:zerobreak}, is a special case of the movement of zero. In fact, let $\textbf{c}=\{c_1,\dots,c_m\}$ be a collection of twin paths based at $P$ as above. For $0<k<m$, given a sub-collection $\{c_{i_1},\dots,c_{i_k}\}$ of adjacent twin paths, the surgery described in Definition \ref{def:move} literally \textit{breaks} $P$ into two lower-order zeros of order $m-k$ and $k$, and only the latter is moved away from $P$. In the local charts, $\varphi$ seen as a branched cover is deformed to a new branched chart $\psi$ that now has two branch values at $\widehat P$ and $\widehat Q$. In this case, $\psi^{-1}(\,\widehat P\,)$ has $m-k+1$ preimages and $\psi^{-1}(\,\widehat Q\,)$ has $k+1$ preimages.
\end{rmk}

\begin{rmk}
    In our setting above, the extremal points $\{Q_o,\dots,Q_m\}$ of $\textbf{c}$ were all supposed to be regular points. This surgery extends \textit{mutatis mutandis} to the case in which some, possibly all, of the $Q_i$'s are also zeros for $\omega$. In this extended setting, two or more zeros of $\omega$ collapse to a new-born zero of higher order. Under this perspective, the movement of zeros is a surgery that not only extends the breaking of a zero but also provide its reverse surgery.
\end{rmk}

\smallskip

\subsection{Handles} Finally, we introduce the following terminology, which we will extensively use in the sequel.

\begin{defn}[Handles and handle-generators]\label{handle} On a closed surface $S_{g}$ of genus $g\ge2$, a \textit{handle} is an embedded subsurface $\Sigma$ that is homeomorphic to $S_{1,1}$, and a \textit{handle-generator} is a simple closed curve that is one of the generators of $\text{H}_1(\Sigma, \mathbb{Z})\cong\shomolzoo$. A \textit{pair of handle-generators} for a handle consists of a pair of simple closed curves $\{\,\alpha, \beta\,\}$ that generate $\text{H}_1(\Sigma, \mathbb{Z})$; in particular, $\alpha$ and $\beta$ intersect once. \end{defn} 

\begin{defn}[Systems of handle generators]\label{def:syshandle}
On a closed surface $S_{g}$ of some positive genus $g\ge2$, we consider a collection of pairwise disjoint $g$ handles $\Sigma_1,\dots,\Sigma_g$. A \textit{system of handle generators} is a collection of $g$ pairs of handle generators $\{\alpha_i,\,\beta_i\}_{1\le i\le g}$ such that $\{\alpha_i,\beta_i\}$ is a pair of handle generators for $\Sigma_i$.
\end{defn}

\noindent We now describe a second surgery introduced by Kontsevich and Zorich in \cite{KZ}. Topologically, \textit{bubbling a handle} consists of adding a handle, say $\Sigma$ to a given surface. Let us explain how this can be done metrically. Let $(X,\omega)\in\mathcal{H}_g(\,\mu\,)$ be a translation surface and let $l\subset (X,\omega)$ be a geodesic segment with distinct endpoints. 

\begin{figure}[!ht] \label{fig:bubbhand}
\centering
\begin{tikzpicture}[scale=1, every node/.style={scale=0.75}]
\definecolor{pallido}{RGB}{221,227,227}
\draw[black, pattern=north west lines, pattern color=pallido] (-3,1) -- (-5, -1) -- (-9,-1) -- (-7, 1) -- (-3,1);
\node[above left] at (-8,0) {$l^-$};
\node[below right] at (-4,0) {$l^+$};
\draw[dashed, black, pattern=north west lines, pattern color=pallido] (0,0) circle (2);
\draw (-1,-1) .. controls (0.1, -0.1) .. (1,1);
\draw (-1,-1) .. controls (-0.1, 0.1) .. (1,1);
\node[above left] at (-0.1, 0.1) {$l^-$};
\node[below right] at (0.1, -0.1) {$l^+$}; 
\fill [white, thin, draw=black] (-1,-1) .. controls (0.1, -0.1) .. (1,1) .. controls (-0.1, 0.1) .. (-1,-1);
\end{tikzpicture}\bigskip
\caption{Bubbling a handle with positive volume.}
\end{figure}
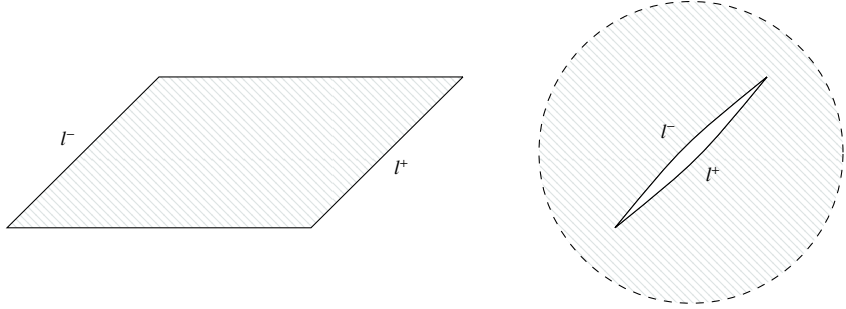

\noindent We slit the translation surface $(X,\omega)$ along $l$ and we label the side that has the surface on its left as $l^+$ and the other side as $l^-$. We then paste the extremal points of the geodesic segment and the resulting surface has two geodesic boundary components of the same length (by construction). Let $\mathcal{P}\subset \mathbb C$ be a parallelogram such that the two opposite sides are both parallel to $l$ (via the developing map) -- the handle $\Sigma$ is obtained by gluing the opposite sides of $\mathcal{P}$. We can paste such a parallelogram to the slit $(X,\omega)$ and the resulting topological surface is homeomorphic to $S_{g+1,n}$. Metrically, we have a new translation surface $(Y,\xi)$. We can distinguish three mutually disjoint possibilities: Let $\mu=(m_1,\dots,m_k)$. 
\begin{itemize}
    \item[1.] Both of the extremal points of $l\subset (X,\omega)$ are regular. Then $(Y,\xi)\in\mathcal{H}_g(2,m_1,\dots,m_k)$.
    \item[2.] One of the extremal points of $l\subset (X,\omega)$ is a zero of $\omega$ of order $m_i$, where $1\le i\le k$. Then $(Y,\xi)$ belongs to the stratum $\mathcal{H}_g(m_1,\dots,m_i+2,\dots,m_k)$.
    \item[3.] Both of the extremal points are zeros of $\omega$ of orders $m_i$ and $m_j$,  respectively. Then $(Y,\xi)$ belongs to the stratum $\mathcal{H}_g(m_1,\dots,\widehat{m}_i,\dots,\widehat{m}_j,\dots,m_i+m_j+2,\dots,m_k)$.
\end{itemize}

\noindent The following holds.

\begin{lem}\label{lem:spininv}
Let $(X,\omega)$ be a translation surface obtained from $(Y,\xi)\in\mathcal{H}_{g}\big(\,2m_1,\dots,2m_k\,\big)$ by bubbling a handle along a slit. Let $2\pi(\,\ell+1\,)$ be the angle around one of the extremal points of the slit, where $\ell=2m_i\ge2$ for some $i=1,\dots,k$ or $\ell=0$. Then, the spin structures determined by $\omega$ and $\xi$ are related as follows: 
\begin{equation}
    \varphi(\omega)-\varphi(\xi)=\ell \,\,\,(\textnormal{ mod }2\, ).
\end{equation}
In particular, bubbling a handle does not alter the spin parity.
\end{lem}

\begin{rmk}\label{rmk:spininv}
It is worth noticing that Lemma~\ref{lem:spininv} differs from \cite[Lemma 11]{KZ}, because the definition of bubbling in our setting is slightly different. In the present paper, a bubbling is performed along a geodesic segment that joins two points, possibly regular. In \cite{KZ}, however, the bubbling of a handle is performed along a saddle connection that joins two zeros obtained after breaking a zero. This preliminary operation is the reason of a possible alteration of the spin parity because, by breaking up a zero, the resulting structure may no longer be of even type. More precisely, if we break up a zero of $\omega$ on a structure of even type, the resulting zeros both have even order or odd order. In the former case, bubbling a handle does not alter the parity. However, in the latter case it does, and \cite[Lemma 11]{KZ} takes into account this possibility. Here we do not break any zero for bubbling a handle, and hence the spin parity remains unaltered in our case.
\end{rmk}

\begin{proof}[Proof of Lemma \ref{lem:spininv}]
As explained in Remark \ref{rmk:spininv}, the desired result follows from \cite[Lemma 11]{KZ}.
\end{proof}

\medskip

\section{Dense and semi-dense representations}\label{sec:notdis}

\noindent In the present section, we aim to prove Theorem \ref{thm:mainthm} for non-discrete representations which are the simplest cases to deal with. We first recall the following terminology for absolute period representations.

\begin{defn}
    A representation $\chi_a\colon {\rm H}_1(S_{g}, \mathbb{Z})\longrightarrow \mathbb C$ is said to be 
    \begin{itemize}
        \item \textit{discrete} if the image of $\chi_a$ is a discrete subgroup of $\mathbb C$. Furthermore, we say that $\chi_a$ is discrete of rank one or \textit{integral} if, up to replacing $\chi_a$ with $A\,\chi_a$ where $A\in\glplus$, then $\textnormal{Im}(\chi_a)=\mathbb Z$. We say that $\chi_a$ is discrete of rank two if, up to replacing $\chi_a$ with $A\,\chi_a$ where $A\in\glplus$, then $\textnormal{Im}(\chi_a)=\mathbb Z\,\oplus\,\mathbb Z\rm i$.
        \item \textit{semi-dense} if, up to replacing $\chi_a$ with $A\,\chi_a$ where $A\in\glplus$, then $\textnormal{Im}(\chi_a)=\mathbb Z\,\oplus\, \rm U\rm{i}$, where $\rm U$ is dense in $\mathbb R$.
        \item \textit{dense} if the image of $\chi_a$ is dense in $\mathbb C$.
    \end{itemize}
\end{defn}

\smallskip

\noindent Note that for the dense and semi-dense representations, our method below works for 
both the holomorphic and the meromorphic cases.

\subsection{Dense absolute representations} We begin by recalling that, according to \cite{BJJP}, for a dense absolute representation $\chia$, its realization in any connected component of a given stratum $\mathcal H_g(\,\mu\,)$ is not subject to any restriction. We can then strengthen these results with the following

\begin{prop}
    Let $\chi\colon \textnormal{H}_1(S_{g},\,Z,\,\mathbb Z)\longrightarrow \mathbb{C}$ be a representation and assume its absolute period representation $\chia$ has dense image in $\mathbb C$. Then $\chi$ can be realized in every connected component of the stratum $\mathcal H_g(\,\mu\,)$ for every partition $\mu=(a_o,\dots,a_k)$ of $2g-2$ where $k+1=|\,Z\,|$.
\end{prop}

\begin{proof}
    We argue by induction on the cardinality of the marked points, \textit{i.e.} on the cardinality $k+1$ of $Z$. The base case $k=0$ holds as a consequence of Remark \ref{rmk:notrirep} along with \cite[Theorem 1.2]{BJJP}. Suppose the claim holds for $|\,Z\,|=k$ and let us show that it holds for $|\,Z\,|=k+1$. Let \(\delta\) be a path joining \(z_o\) and \(z_1\). Since $\chia$ is dense, we can choose a suitable representative of $\delta$ such that  $|\,\chi(\,\delta\,)\,|$ is as small as we want. By induction, we can first realize $\chi$ restricted to any connected component of the stratum $\mathcal{H}_g(a_o+a_1, a_2, \ldots, a_k)$, and then break up the zero of order $a_o+a_1$, as in \S\ref{sec:zerobreak}, to get a zero of order $a_o$ at the same location with a nearby newborn zero of order $a_1$ such that the saddle connection joining them gives the desired (short) relative period $\chi(\,\delta\,)$. Note that this operation does not affect the other absolute and  relative periods.
\end{proof}


\subsection{Dense absolute representations in $\mathbb{Z} \times \mathbb{R}{\rm i}$} In the present subsection, we assume that the image of $\chia$ spans $\mathbb{Z}$ in the horizontal direction and is dense in the vertical direction. Once again, according to \cite{BJJP}, for a dense absolute representation $\chia$ its realization in any connected component of a given stratum $\mathcal H_g(\,\mu\,)$ is not subject to any restriction. We can then strengthen these results with the following 

\begin{prop}
    Let $\chi\colon \textnormal{H}_1(\,S_{g},\,Z,\,\mathbb Z\,)\longrightarrow \mathbb{C}$ be a representation and assume its absolute period representation $\chia$ is semi-dense. Then $\chi$ can be realized in every connected component of the stratum $\mathcal H_g(\,\mu\,)$ for every partition $\mu=(a_o,\dots,a_k)$ of $2g-2$ where $k+1=|\,Z\,|$.
\end{prop}

\begin{proof}
    Up to replacing $\chi$ with another representation $A\cdot\chi$, where $A\in\text{GL}^+(2,\mathbb R)$, we can assume that the image of $\chia$ to be dense in $\mathbb Z\,\oplus\mathbb R\rm i$. Let $(X_o,\omega_o)\in \mathcal{H}_g(\,2g-2\,)$ be a translation surface with absolute period character $\chia$ and let $z_o$ be the unique zero of $\omega$. Define $\varepsilon=\varepsilon\,(\,X_o,\omega_o,z_o\,)$ as follows: 
    \begin{equation}
        \varepsilon= \sup\Big\{\, \rho\in\mathbb R^+\,\big|\, \textnormal{ for any }\,x,y\in B_\rho(\,z_o\,) \textnormal{ there is a unique geodesic joining them }\,\Big\}.
    \end{equation}
    
    \noindent Notice that there is a simple closed curve whose absolute period has modulus equal to $2\varepsilon$. Now, up to absolute periods, we can assume that the relative period $r_i$ of the path joining $z_o$ to $z_i$ satisfies that the real part of $r_i$ is bounded strictly by $\varepsilon$ and the imaginary part of $r_i$ can be chosen arbitrarily up to an arbitrarily small error term. Order the zeros $z_1, \ldots, z_n$ according the real part distances of $r_i$ to $z_o$, from near to far. If some of them have the same real part, perturb their imaginary parts a little bit by using absolute periods, and hence we can assume  
    \begin{equation} 
        0 < |r_1| < |r_2| < |r_3| < \cdots < |r_k| < \varepsilon 
    \end{equation} 
    while keeping their imaginary parts small. First on $(X_o,\omega_o)$, we split the unique zero in a neighborhood of radius less than $\varepsilon$ so that the newborn zeros are joined by a relative period $r_k$, \textit{i.e.} the resulting translation surface $(X_1,\omega_1)\in\mathcal{H}_g(\,2g-2 - a_k, a_k\,)$, where we used the assumption $|r_k| < \varepsilon$. Next, we split the newborn zero of order $2g-2-a_k$ in a smaller neighborhood of radius between $|r_{k-1}|$ and $|r_k|$ to two zeros with relative period $r_{n-1}$ and the resulting structure belongs to $\mathcal{H}_g(\,2g-2 - a_{k-1} - a_{k},\, a_{k-1},\, a_k\,)$, where we used the assumption $|r_{k-1}| < |r_k|$. Iterating this process, we thus obtain a translation surface $(X_k,\omega_k)\in\mathcal H_g(\,\mu\,)$ with the desired relative period representation.  
\end{proof}

\bigskip

\section{Square-tiled surfaces with prescribed relative periods}\label{sec:sqrspesper}

\noindent We now consider a representation $\chi\colon \textnormal{H}_1(S_{g,n},\,Z,\,\mathbb Z)\longrightarrow \mathbb{C}$ such that its absolute period representation $\chia$ has discrete image in $\mathbb C$. Under this assumption, Haupt's Theorem in \cite{OH} states that $\chia$ can be realized as the period character of some translation surface if its volume is strictly positive and hence the image of $\chia$ spans a lattice $\Lambda$ in $\mathbb C$. Moreover, we assume without loss of generality that $\Lambda=\ziz$. In what follows, let $T$ denote the topological torus and let $\mathbb T$ denote the flat torus $\mathbb C\,/\,\ziz$. In the present section, we only consider the holomorphic case. It is customary to define a translation surface $(X,\omega)$ as square-tiled surface if there exists a branch covering, say $\pi\colon X\longrightarrow \mathbb T$, with a unique branch value such that $\omega=\pi^*dz$. In the present work, however, we need to extend this definition as follow.

\begin{defn}[Square-tiled surface]\label{defn:transurf}
    A translation surface $(X,\omega)$ is defined as a \textit{square-tiled surface} if there exists a branch covering $\pi\colon X\longrightarrow \mathbb T$ such that $\omega=\pi^*dz$. A translation surface is defined to be an \textit{origami} if $\pi$ is branched at a single point.
 \end{defn}

\noindent In other words, in our setting, for a translation surface $(X,\omega)$, its zeros do not need to map to the same branch point of $\mathbb T$. See also Remark \ref{rmk:STS}. As a consequence, we have the following

\begin{lem}
    A translation surface is an origami if and only if its relative periods are all integral in $\ziz$.
\end{lem}

\noindent The proof directly follows from our definitions.


\smallskip

\subsection{A necessary condition.}\label{ssec:neccond} Before stating the main result of the present section, we begin by stating the necessary condition for a representation $\chi$ to be realized in a given stratum. Let $\chia\colon \textnormal{H}_1(S_{g},\,\mathbb Z)\longrightarrow \mathbb{C}$ be a discrete representation such that $\text{Im}(\chia)=\ziz$. 
The realizable cases thus correspond to square-tiled surfaces as defined in Definition \ref{defn:transurf}, where ${\rm vol}(\chi)$ is the degree of the corresponding branched cover of the square torus. 

\medskip

\noindent As before, denote by $\Lambda$ the lattice $\mathbb{Z} + \mathbb{Z}{\rm i}$. Decompose $Z = \{z_o, \ldots, z_k\}$ into maximum disjoint subsets, say $Z_o, \ldots, Z_{\ell}$, such that for any $\gamma$ joining $z_{i_1}, z_{i_2}\in Z_i$ the image $\chi(\gamma) \in \Lambda$. Then $\chi$ arises as the period character of some pair $(X,\omega)$ of type $\mu = (a_o, \dots, a_k)$ if and only if there exists a primitive branched cover $\pi\colon X\longrightarrow \mathbb{C}/\Lambda$ of degree $d = {\rm vol}(\,\chi\,)$ such that $\pi$ has $\ell+1$ branch points $y_o, \ldots, y_\ell \in \mathbb{C}/\Lambda$ and over each $y_i$ the ramification points are the zeros in $Z_i$ with ramification orders given by the zero orders. Without loss of generality, we can choose the positions of $y_o, \ldots, y_\ell$ in the unit square such that any two of them have distinct horizontal and vertical coordinates.  

\smallskip

\noindent Let $\mu^{\vdash} = (\mu_o;\, \dots;\, \mu_\ell)$ be a partition of a given signature $\mu$ into $\ell + 1$ disjoint subsets, where each tuple $\mu_i$ records the orders of the zeros in $Z_i$. Let $|\,\mu_i\,|$ be the sum of the entries and $|\,Z_i\,|$ be the order of the subset. Define 
\begin{equation}\label{eq:neccondtoreal}
    {\rm m}(\,\chi\,,\,\mu^\vdash\,) = \max_{0\le i\le\ell} \Big\{\, |\,\mu_i\,| + |\,Z_i\,|\,\Big\}.
\end{equation}

\noindent As a consequence, the above description implies the following necessary condition for a period character $\chi$ to be realizable: 
\begin{eqnarray}
\label{eq:hol-necessary}
 \deg(\,\pi\,)={\rm vol}(\,\chia\,) \geq {\rm m}(\,\chi\,,\, \mu^\vdash\,) 
 \end{eqnarray}
by considering the degree of $\pi$. 

\medskip

\textit{Convention.} In order to avoid any ambiguity, in what follows we write $\mu^\vdash=\mu$ if and only if each element of the partition is a singleton; that is $\mu_i=\{\,a_i\,\}$ for all $i=0,\dots, k$. In this case,  we say that $\mu$ is \textit{fully partitioned}. On the other extreme, we say that $\mu$ is \textit{not partitioned} if $\ell=0$. We call each element $\mu_i$ of $\mu^\vdash$ a sub-signature.

\medskip

\begin{rmk}\label{rmk:shcd}
    For an absolute period representation $\chia\colon\shomolz\longrightarrow \mathbb C$, in \cite[Theorem 1.1]{fils} and \cite[Theorem 1.1]{BJJP} the authors have independently showed that if the image $\text{Im}(\,\chia\,)=\Lambda$ is a lattice,  then $\chia$ can be realized in the stratum $\hm$ if and only if 
    \begin{equation}\label{eq:hol-necweaker}
        {\rm vol}(\,\chia\,) \ge \max_{0\le i\le k} \big(\,a_i+1\,\big)\,{\rm area}(\,\mathbb C/\Lambda\,)
    \end{equation}
    where $\mu=(a_o,\dots,a_k)$. Notice that inequality \eqref{eq:hol-necessary} turns out to be a stronger condition than \eqref{eq:hol-necweaker} and they provide the same constraint if and only if $Z$ decomposes into $k+1$ singletons, {\em i.e.} $\mu_i=\{\,a_i\,\}$ for all $i$. This happens because prescribing relative periods generally imposes further conditions on the volume, and hence on the degree of the branched covering. In fact, given $\chi\colon \textnormal{H}_1(S_{g},\,Z,\,\mathbb Z)\longrightarrow \mathbb{C}$, the cardinality of $|\,Z\,|$ determines the length of the signature $\mu$, but it does \textit{not} necessarily determine a partition $\mu^\vdash$. As a direct consequence, the value of ${\rm m}(\,\chi\,,\,\mu^\vdash\,)$ strongly \textit{depends} on $\mu^\vdash$ whereas the volume only depends on $\chia$. In particular, square-tiled surfaces with the lowest possible volume have no (strictly) relative periods in $\Lambda$. 
\end{rmk}

\begin{rmk}\label{rmk:relperint}
    It is worth recalling that $\mathbb T$ admits a natural group structure, with neutral element $0_{\mathbb T}$, inherited by the additive structure on $\mathbb C$. Moreover, such an operation is compatible with the metric structure $d_{\mathbb T}$ induced by the Euclidean distance in $\mathbb C$. In fact,
    \begin{equation}
        d_{\mathbb T}(\,p_1,\,p_2\,)=|\,p_1-p_2\,|.
    \end{equation}
    It is not hard to see that $\big(\,\mathbb T,\,d_{\mathbb T}\,\big)$ turns out to be a geodesic space, but it fails to be a \textit{uniquely} geodesic space, see \cite[Chapter I.1, page 4]{BH}. As a consequence, since relative periods with the same endpoints differ by an element in $\ziz$, we can regard a relative period as a well-defined element of $\mathbb T$. 
\end{rmk}


\subsection{Realization of branched covers} Next,  we will show that the necessary condition described in \S\ref{ssec:neccond} is also sufficient, \textit{i.e.} we aim to show the following

\begin{prop}\label{prop:branchcoverings}
    Let $\mu$ be a signature of genus $g$ holomorphic differentials and let $\mu^\vdash$ be a partition of $\mu$. For each integer $d \ge {\rm m}(\,\chi\,,\mu^\vdash\,)$, there exists a primitive branched covering $\pi\colon X\longrightarrow  \mathbb C/\ziz$ of degree $d$ such that $(X, \,\pi^*dz)$ belongs to $\mathcal H_g(\,\mu\,)$.
\end{prop}

\noindent As a direct consequence, the following holds. 

\begin{cor}
    Let $\chi\colon\textnormal{H}_1(\,S_{g},\,Z,\,\mathbb Z\,)\longrightarrow \mathbb C$ be a representation and let $|\,Z\,|=k+1$. Let $\mu=(a_o,\dots,a_k)$ be a partition of $2g-2$. Then $\chi$ can be realized in a connected component of $\hm$ if and only if $\chia$ can be realized in the same connected component.
\end{cor}

\noindent 

\noindent In what follows we also need to include the case where some $a_i = 0$, that is, ordinary marked points as zeros of order zero. We first consider the following remark as a premise. 

\begin{rmk}
    Let $g=1$ and $\mu = (0,\ldots,0)$. Line up $d$ unit squares to get a $d\times 1$ rectangle and identify the two pairs of parallel edges of the rectangle to form a torus $X$ of volume $d$. For any $i=0,\dots,\ell$, set $d_i=|\,Z_i\,|$ and for each subset $Z_i=\{z_{i,1}, z_{i,2}, \ldots, z_{i,\,d_i} \}$ of the marked points that map to the same image $y_i$, choose and label the points from the first $|\,Z_i\,|$ unit squares of $X$. Note that \eqref{eq:hol-necessary} ensures that there are enough squares to use.
\end{rmk}

\medskip

\noindent Let us assume $g\ge2$. For a signature $\mu$, we will discuss the realization of $\chi\colon \textnormal{H}_1(S_{g},\,Z,\,\mathbb Z)\longrightarrow \mathbb{C}$ according to the following three mutually disjoint cases:
\begin{itemize}
    \item[1.] the stratum $\hm$ is connected -- see Section \S\ref{ssec:sqtsconnstra},
    \item[2.] no relative period is integral, \textit{i.e.} $\mu^\vdash=\mu$ -- see Section \S\ref{ssec:nointrelper}, or
    \item[3.] the stratum $\hm$ is not connected and $\mu^\vdash\,\neq\, \mu$. In this latter case our proof is given in Sections \S\ref{ssec:lastcasehyp} and \S\ref{ssec:lastcasespintwo}. 
\end{itemize}

\noindent The first two cases are easier to handle and essentially follow from the previous result. The last case, instead, requires more work.

\subsection{Connected strata}\label{ssec:sqtsconnstra} Let $\mu$ be a signature of genus $g$ holomorphic differentials such that the corresponding stratum $\hm$ is connected. According to \cite{KZ}, most of the strata fall into this range. In this case, the realization of $\chi$ into $\hm$ is subject to the realization of a primitive branched cover $\pi\colon S_g\longrightarrow T$ with prescribed branching data. The realization problem has been studied \cite{EKS}; however the problem of realizing primitive branched covers had been solved later in \cite{BGKZ} and \cite{BGKZ2}. We proceed as follows.

\bigskip

\noindent Let $\chi\colon\textnormal{H}_1(S_{g},\,Z,\,\mathbb Z)\to \mathbb{C}$ be a representation with absolute period representation $\chia$. Let $\mu=(a_o,\dots,a_k)$ be a signature of genus $g$ holomorphic differentials such that $|\,Z\,|=k+1$. As above, $\chi$ naturally determines the partition of $Z$ into $\ell+1$ disjoint subsets $Z_i$ of cardinality $r_i$. Let $\mu^\vdash=(\mu_o;\dots;\mu_{\ell})$ be any partition of $\mu$ such that $\mu_i=\left(\,a_{i\,1},\dots,a_{i\,r_i}\,\right)$ has cardinality $r_i$. Let $d={\rm m}(\,\chi\,,\,\mu^\vdash\,)$ and, without loss of generality, we can assume that 
\begin{equation}
    d\,=\,|\,\mu_1\,|\,+\,|\,Z_1\,|= \sum_{j=1}^{r_1} a_{1j}\,+\,r_1=\sum_{j=1}^{r_1} \big(\,a_{1j}\,+\,1\,\big).
\end{equation}

\noindent We now define a \textit{branch data}, namely we use the partition $\mu^{\vdash}$ to define the necessary algebraic conditions, provided by Riemann--Hurwitz's formula, to the existence of a branched cover. Let $B_o=\left [a_{01}+1,\dots,a_{0r_1}+1 \right]$. For all $i=1,\dots,\ell$, we set

\begin{equation}
    B_i=\left [a_{i1},\dots,a_{ir_i}, 1,\dots,1 \right]
\end{equation}

\noindent where the entry $1$ appears $d-\left(\,|\,\mu_i\,|\,+\,r_i\, \right)$ times.

\begin{rmk}
    This is equivalent to extend $\mu_i=\left(\,a_{i\,1},\dots,a_{i\,r_i}\,\right)$ to a partition $\left(\,a_{i\,1},\dots,a_{i\,r_i}\,0,\dots,0\,\right)$
    where the entry $0$ appears $d-\left(\,|\,\mu_i\,|\,+\,r_i\, \right)$ times. Moreover, the following identity now holds:
    \begin{equation}
        d\,= \sum_{j=1}^{r_i} \big(\,a_{ij}\,+\,1\,\big)+\,\Big(\, d-(\,|\,\mu_i\,|\,+\,r_i\,)\,\Big)\,.
    \end{equation}
\end{rmk}

\noindent The datum $\mathcal B=[B_o,\dots,B_{\ell}]$ thus defines a branch data and \cite{BGKZ} ensures the existence of a primitive branched cover, say $\pi\colon S_g\longrightarrow T$ with branch data $\mathcal B$. Notice that by turning $T$ into $\mathbb C/\ziz$, this argument would be enough to realize $\chia$ in the stratum $\hm$ by pulling back the translation structure of $\mathbb T$ on $S$. In order to realize $\chi$ in $\hm$, we also need to prescribe the branched values, say $y_o,\dots,y_\ell$. For this purpose, let us consider a subset $\{\,z_o,\dots,z_\ell\,\}\subset Z$ such that $z_i$ is a representative of $Z_i$ for every $i=0,\dots,\ell$. Since any one of these points can be equally entitled as the base point,  we fix it as $z_o$ for simplicity. For $i=1,\dots,\ell$, let $\delta_i$ be a path joining $z_o$ with $z_i$. It clearly follows from the assumptions that
\smallskip
\begin{itemize}
    \item[(i)] $\chi(\,\delta_i\,)=r_i\notin\,\ziz\,$ for every $i=1,\dots,\ell$, and
    \item[(ii)] $\chi(\,\delta_i\delta_j^{-1}\,)=\chi(\,\delta_i\,)-\chi(\,\delta_j\,)\notin \ziz\,$ for every $i,j \in\{1,\dots,\ell\}$ and $i\neq j$.
\end{itemize}
\smallskip
\noindent Let $y_o=\pi(\,z_o\,)$ and define $y_i=y_0\,+\,r_i$. As a direct consequence of Remark \ref{rmk:relperint}, $y_i\in\mathbb T$ is well-defined. Observations (i) and (ii) above guarantee that $y_i\neq y_j$ for all $i,j\in\{0,\dots,\ell\}$. Therefore, by pulling back the differential $dz$ from $\mathbb C/\ziz$ via $\pi$, we obtain a square-tiled surface in the stratum $\hm$ with the desired relative periods.

\begin{rmk}
    Notice that the same argument shows that $\chi$ can be realized in every stratum. Therefore, the next sections focus on the realization in connected components of strata whenever they are disconnected.
\end{rmk}


\smallskip

\subsection{No integral relative periods}\label{ssec:nointrelper}

\noindent Next, we focus on signatures $\mu=(a_o,\dots,a_k)$ of genus $g$ holomorphic differentials such that $\hm$ has at least two connected components. Let $\chi$ be a representation such that no relative period is integral. This is equivalent to saying that $\mu^{\vdash}=\mu$ and hence ${\rm m}(\,\chi,\,\mu^{\vdash}\,)=\max \big(\,a_i+1\,)$, comparing \eqref{eq:neccondtoreal} and \eqref{eq:hol-necweaker}. According to \cite[Proposition 1.3]{BJJP}, for any $d\ge {\rm m}(\,\chi,\,\mu^{\vdash}\,)$, we can realize the absolute period representation $\chia$ in every connected component of $\hm$. Therefore, it remains to prescribe the relative periods. For this purpose, we will use Schiffer variations, \textit{i.e.} movement of branched points, see Paragraph \S\ref{sssec:schiffer}.

\medskip

\noindent For this purpose, let $(X,\omega)$ be a translation surface realized in a prescribed connected component of $\hm$ as done in \cite{BJJP}. By construction, such a structure comes with a branched covering map $\pi\colon X\longrightarrow \mathbb T$ with branch values at $\{\,y_o,\,\dots,y_k\,\}$. Since no pair of zeros has integral period, it follows that every $y_i$ has only one zero, say $z_i$, in its fiber. Denote by $z_o$ the sole zero in the fiber of $y_o$ and fix it as the base point. For every $i=1,\dots,k$, let $\delta_i$ be a path that joins $z_o$ with $z_i$ and let $r_i$ be the corresponding relative period. 

\smallskip

\noindent By Remark \ref{rmk:relperint}, $y_i=y_o\,+\,r_i$ holds and hence, in order to prescribe the relative periods, we need to deform the branched covering $\pi$ so that $y_i$ satisfies in $\mathbb T$ the equation
\begin{equation}\label{eq:condcover}
   y_i=y_o\,+\, \chi(\,\delta_i\,)
\end{equation}

\noindent for each $i=1,\dots,k$. We will deform $\pi$ by using the movement of zeros, see Section \S\ref{sssec:schiffer}. Let $\widetilde{\mathbb T}$ be the universal cover of $\mathbb T$, with covering projection $p\colon\widetilde{\mathbb T}\longrightarrow \mathbb T$, and fix a conformal identification $\widetilde{\mathbb T} \simeq \mathbb C$ whereby $\pi_1(\mathbb T)$ acts on $\mathbb C$ via translations and such that $p^{-1}(y_o)=\ziz$. Two possible situations can occur. Either
\begin{itemize}
    \item[1.] $p^{-1}(y_i)=\chi(\,\delta_i\,)\,+\,\ziz$; that means $r_i=\chi(\,\delta_i\,)$. In this case, $z_i$ is well located and we do not need to move it elsewhere in $(X,\omega)$, or
    \smallskip
    \item[2.] $r_i\neq \chi(\,\delta_i\,)$. In this latter case, we need to move $z_i$ as in Definition \ref{def:move} so that the resulting relative period is as desired. 
\end{itemize}

\noindent Suppose we are in the second case. Since $r_i\neq \chi(\,\delta_i\,)$, it readily follows that $\chi(\,\delta_i\,)\,+\,\ziz$ and $r_i\,+\,\ziz$ are disjoint discrete subsets of $\mathbb C$. In particular, there are two points, say $\widetilde{y}_i\in r_i\,+\,\ziz$ and $\widetilde{p}_i\in \chi(\,\delta_i\,)\,+\,\ziz$ that realize the distance between these subsets -- notice that such a distance is strictly positive because both sets are discrete and disjoint.

\smallskip

\noindent Let $\widetilde{c}_i\subset \mathbb C$ be the geodesic segment joining $\widetilde{y}_i$ and $\widetilde{p}_i$ and let $\widehat{c_i}$ be its projection in $\mathbb T$. Clearly, $\widehat{c}_i$ is a geodesic segment in $\mathbb T$ for the induced distance $d_{\mathbb T}$, see Figure \ref{fig_branchpoints}. Notice that $\widehat{c}_i$ being a segment joining $y_i$ and $y_i+r_i$ has relative period $r_i$ by construction.

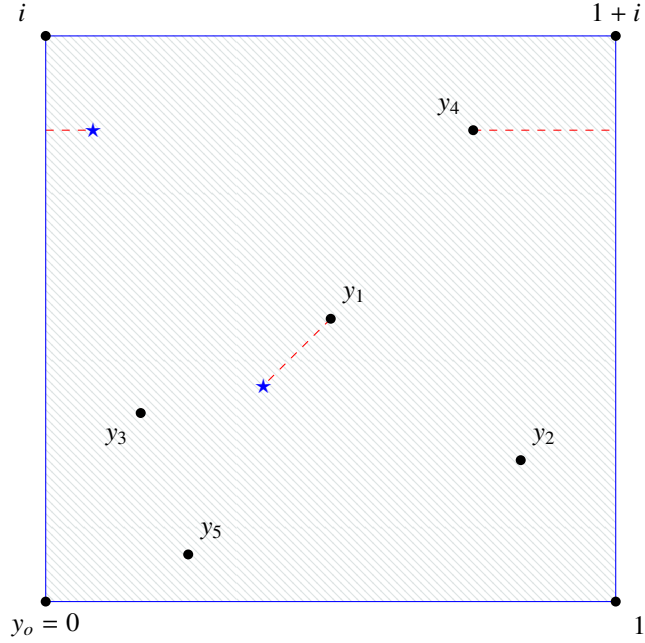
\begin{figure}[h!]
    \centering
    \begin{tikzpicture}[scale=1.25, every node/.style={scale=1}]
    \definecolor{pallido}{RGB}{221,227,227}
    
    \pattern [pattern=north west lines, pattern color=pallido]
    (-3,-3)--(3,-3)--(3,3)--(-3,3)--(-3,-3);

    \draw [red, dashed] (0,0) to (-0.707,-0.707);
    \draw [red, dashed] (1.5,2) to (3,2);
    \draw [red, dashed] (-3,2) to (-2.5, 2);

    \draw [blue] (-3,-3) to (3,-3);
    \draw [blue] (3,-3)  to (3,3);
    \draw [blue] (3,3)   to (-3,3);
    \draw [blue] (-3,3)  to (-3,-3);

    \fill (-3,-3) circle (1.5pt);
    \fill (3,-3) circle (1.5pt);
    \fill (3,3) circle (1.5pt);
    \fill (-3,3) circle (1.5pt);

    \fill (0,0) circle (1.5pt);
    \fill (2,-1.5) circle (1.5pt);
    \fill (-2,-1) circle (1.5pt);
    \fill (1.5,2) circle (1.5pt);
    \fill (-1.5,-2.5) circle (1.5pt);
    
    \node at (-3,-3.25) {$y_o=0$};
    \node at (-3.25,3.25) {$i$};
    \node at (3.25,-3.25) {$1$};
    \node at (3,3.25) {$1+i$};
    
    \node at (0.25,0.25) {$y_1$};
    \node at (2.25,-1.25) {$y_2$};
    \node at (-2.25,-1.25) {$y_3$};
    \node at (1.25,2.25) {$y_4$};
    \node at (-1.25,-2.25) {$y_5$};

    \node[blue] at (-0.707,-0.707) {$\star$};
    \node[blue] at (-2.5,2) {$\star$};
   
    \end{tikzpicture}
    \caption{Location of branch values on $\mathbb T$. The blue stars denote where we want to move the zeros and the red dashed lines denote the segments $\widehat{c}_i$ to be lifted to obtain the twins needed to perform the movement. }
    \label{fig_branchpoints}
\end{figure}

\smallskip

\noindent We use the covering $\pi$ to lift $\widehat{c_i}$ to a collection $\{\,c_{i,1},\dots,c_{i,a_i+1}\,\}$ of $a_i+1$ pairwise twin paths based at $z_i$. In fact, the lifting path property fails at every ramification point $z_i$ of $\pi$. Suppose for a moment that every pair of twin paths intersects only at $z_i$. We thus perform a movement of zeros to deform $\pi$ to a new branched covering having $r_i$ as the branch value in place of $y_i$. By repeating this operation for every index such that $r_i\neq\widetilde{y}_i$, after a finite number of steps, we obtain a translation surface $(Y,\xi)$ with the desired relative periods and invariants.

\smallskip

\noindent To conclude, the following two things remain to be addressed. Firstly, we have assumed that every pair of twins intersects only at $z_i$. We show that this is always the case. 
The key point is that the extremal point of $\widehat{c}_i$ other than $y_i$ must be a regular value. In fact, if two twins based at $z_i$ intersect at any other point, then this latter point is necessarily a zero, say $z_j$, that projects to $y_j$. This leads to the desired contradiction because $r_i$ is supposed to be a regular value.

\smallskip

\noindent The second remark is that a segment $\widehat{c}_i$ can contain in its interior a branch value $y_j$. In this case, one lift of $\widehat{c}_i$ might pass through the zero $z_j$. We bypass this issue by moving $z_j$ a little bit so that $\widehat{c}_i$ does no longer contain $y_j$ after the movement -- in other words, we make $y_j$ a regular value. We next move $z_i$ as needed and, finally, we put $z_j$ back to make $y_j$ a branch value again. 

\smallskip

\subsection{Realization in the hyperelliptic components of strata}\label{ssec:lastcasehyp}

It remains to consider the realization of representations in the connected components for disconnected strata. From now on, we assume that at least one relative period is integral. Recall that a stratum $\hm$ is called \textit{minimal} if $\mu=(2g-2)$ and not minimal otherwise. We begin with realizing translation surfaces in the hyperelliptic components.

\smallskip

\noindent We first notice that, according to Kontsevich--Zorich's classification of the connected components of strata in \cite{KZ}, a \textit{non-minimal} stratum $\hm$ admits a hyperelliptic component if and only if $\mu$ is a signature of type $(g-1,g-1)$. Since the case $\mu^\vdash=\mu$ is already covered in Section \S\ref{ssec:nointrelper}, throughout this section we always assume $\mu^\vdash\neq\mu$. It follows that 
\begin{equation}
    {\rm m}(\,\chi,\,\mu^{\vdash}\,)= 2g,
\end{equation}

\noindent according to \eqref{eq:neccondtoreal}. For this purpose, we will realize $(X,\omega)$ as a translation surface in the desired connected component of $\hm$ by realizing a branched covering map $\pi\colon X\longrightarrow \mathbb T$ with two ramification points over a single branch value. As the name suggests, a square-tiled surface has a natural tiling with the unit square as a prototype. 

\subsubsection{Realization process in the hyperelliptic component}\label{sssec:hypercomp} Let $\mathcal R$ be a rectangle of size $g\,\times\,2$. By gluing the opposite sides via translation, we obtain a flat torus $T$ of volume $2g$ that comes with a genuine covering map $\pi\colon T\longrightarrow \mathbb T$ of degree $2g$ by construction. Let $q\in\mathbb T$ be any point and let $Q$ denote the set $\{\,q_{1\,1},\dots,q_{1\,2g},q_{2\,1},\dots,q_{2\,2g}\,\}$ of preimages of $q$ via the covering $\pi$. The labels are given so that each point can be identified as follows: $q_{i\,j}$ is the preimage at the $i-$th row from the bottom and at the $j-$th column from the left, see Figure \ref{fig:hypcasebase}.

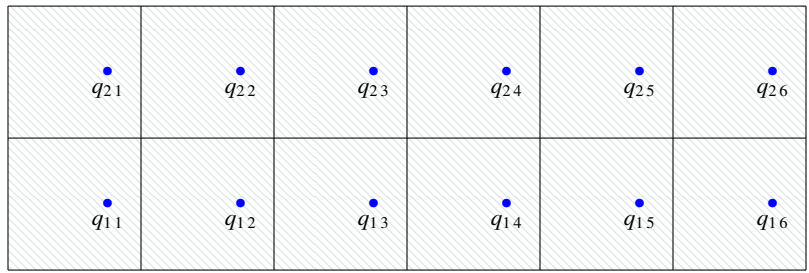
\begin{figure}[h!]
    \centering
    \begin{tikzpicture}[scale=0.875, every node/.style={scale=0.875}]
    \definecolor{pallido}{RGB}{221,227,227}

    \pattern [pattern=north west lines, pattern color=pallido]
    (0,0)--(12,0)--(12,4)--(0,4)--(0,0);
    
    \draw[thin, black] (0,0)--(12,0);
    \draw[thin, black] (0,2)--(12,2);
    \draw[thin, black] (0,4)--(12,4);
    \foreach \x in {0, 2, 4, 6, 8, 10, 12} {
        \draw (\x, 0) -- (\x, 4);
        }

    \foreach \x in {0, 2, 4, 6, 8, 10} {
        \node[blue] at (\x+1.5, 1) {$\bullet$};
        \node[blue] at (\x+1.5, 3) {$\bullet$};
        }

    \node at ( 1.5,0.725)      {$q_{1\,1}$};
    \node at ( 3.5,0.725)      {$q_{1\,2}$};
    \node at ( 5.5,0.725)      {$q_{1\,3}$};
    \node at ( 7.5,0.725)      {$q_{1\,4}$};
    \node at ( 9.5,0.725)      {$q_{1\,5}$};
    \node at (11.5,0.725)      {$q_{1\,6}$};

    \node at ( 1.5,2.725)      {$q_{2\,1}$};
    \node at ( 3.5,2.725)      {$q_{2\,2}$};
    \node at ( 5.5,2.725)      {$q_{2\,3}$};
    \node at ( 7.5,2.725)      {$q_{2\,4}$};
    \node at ( 9.5,2.725)      {$q_{2\,5}$};
    \node at (11.5,2.725)      {$q_{2\,6}$};

    \end{tikzpicture}
    \caption{The rectangle of size $6\,\times\,2$ along with a collection of points arising as the preimages of some point $q\in\mathbb T$.}
    \label{fig:hypcasebase}
\end{figure}

\noindent Notice that there is an involution $\tau$ that swaps the two rows and the $j-$th column with the $(g-j)-$th column. Additionally, if $g$ is odd, the middle column is mapped to itself. It is easy to verify that there are four fixed points. We now perform some surgeries to produce a structure with the desired properties. In the first place, join the adjacent points $q_{1\,j}$ and $q_{1\,j+1}$ with an edge $e_{1\,j}$ for all $j=1,\dots,g-1$. In the same fashion, join the adjacent $q_{2\,j}$ and  $q_{2\,j+1}$ with an edge $e_{2\,g-j}$ for all $j=1,\dots,g-1$; see Figure \ref{fig:hypcasebase2}.

\begin{figure}[h!]
    \centering
    \begin{tikzpicture}[scale=0.875, every node/.style={scale=0.875}]
    \definecolor{pallido}{RGB}{221,227,227}

    \pattern [pattern=north west lines, pattern color=pallido]
    (0,0)--(12,0)--(12,4)--(0,4)--(0,0);
    
    \draw[thin, black] (0,0)--(12,0);
    \draw[thin, black] (0,2)--(12,2);
    \draw[thin, black] (0,4)--(12,4);
    \foreach \x in {0, 2, 4, 6, 8, 10, 12} {
        \draw (\x, 0) -- (\x, 4);
        }

    \draw[thin, orange] ( 1.5,1)--(11.5,1);
    \draw[thin, orange] ( 1.5,3)--(11.5,3);

    \foreach \x in {0, 2, 4, 6, 8, 10} {
        \node[blue] at (\x+1.5, 1) {$\bullet$};
        \node[blue] at (\x+1.5, 3) {$\bullet$};
        }

    \node at ( 2.5, 1.25)      {$e_{1\,1}$};
    \node at ( 4.5, 1.25)      {$e_{1\,2}$};
    \node at ( 6.5, 1.25)      {$e_{1\,3}$};
    \node at ( 8.5, 1.25)      {$e_{1\,4}$};
    \node at (10.5, 1.25)      {$e_{1\,5}$};

    \node at ( 2.5, 3.25)      {$e_{2\,5}$};
    \node at ( 4.5, 3.25)      {$e_{2\,4}$};
    \node at ( 6.5, 3.25)      {$e_{2\,3}$};
    \node at ( 8.5, 3.25)      {$e_{2\,2}$};
    \node at (10.5, 3.25)      {$e_{2\,1}$};

    \node at ( 1.5,0.725)      {$q_{1\,1}$};
    \node at ( 3.5,0.725)      {$q_{1\,2}$};
    \node at ( 5.5,0.725)      {$q_{1\,3}$};
    \node at ( 7.5,0.725)      {$q_{1\,4}$};
    \node at ( 9.5,0.725)      {$q_{1\,5}$};
    \node at (11.5,0.725)      {$q_{1\,6}$};

    \node at ( 1.5,2.725)      {$q_{2\,1}$};
    \node at ( 3.5,2.725)      {$q_{2\,2}$};
    \node at ( 5.5,2.725)      {$q_{2\,3}$};
    \node at ( 7.5,2.725)      {$q_{2\,4}$};
    \node at ( 9.5,2.725)      {$q_{2\,5}$};
    \node at (11.5,2.725)      {$q_{2\,6}$};
    
    \end{tikzpicture}
    \caption{The rectangle of size $6\,\times\,2$ along with a collection of points arising as the preimages of some point $q\in\mathbb T$. Adjacent points are joint with an edge $e_{i\,j}$.}
    \label{fig:hypcasebase2}
\end{figure}
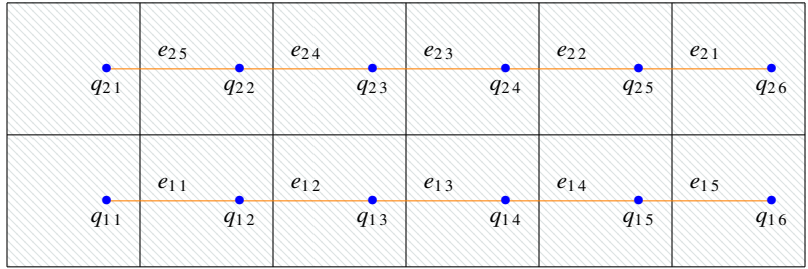

\noindent We then slit all edges $e_{i\,j}$ and denote the resulting sides as $e_{i\,j}^{\pm}$ according to our convention -- without loss of generality we can assume that all edges are oriented from the left to the right. Re-glue them as follows: $e_{1\,j}^+$ is identified with $e_{2\,j}^-$ and, similarly, $e_{1\,j}^-$ is identified with $e_{2\,j}^+$. The resulting structure is a translation surface with signature $(g-1,g-1)$; see Figure \ref{fig:hypcasebase3} -- hence, it has genus $g$ -- and its absolute and relative periods all belong to $\ziz$ as desired. 

\begin{figure}[h!]
    \centering
    \begin{tikzpicture}[scale=0.875, every node/.style={scale=0.875}]
    \definecolor{pallido}{RGB}{221,227,227}

    \pattern [pattern=north west lines, pattern color=pallido]
    (0,0)--(12,0)--(12,4)--(0,4)--(0,0);
    
    \draw[thin, black] (0,0)--(12,0);
    \draw[thin, black] (0,4)--(12,4);
    \foreach \x in {0, 12} {
        \draw (\x, 0) -- (\x, 4);
        }

    \foreach \x in {0, 2, 4, 6, 8, 10} {

    \draw[thin, black] ( \x+1.5,1) arc [start angle = 0, end angle =360,radius = 0.5];
    \draw[thin, black] ( \x+1.5,3) arc [start angle = 0, end angle =360,radius = 0.5];
    }

    \foreach \x in {0, 4, 8} {

    \draw [fill=orange, opacity=0.25] ( \x+1,1) circle (0.5);
    \draw [fill=orange, opacity=0.25] ( \x+1,3) circle (0.5);
    }

    \foreach \x in {0, 4, 8} {

    \draw [fill=yellow, opacity=0.25] ( \x+3,1) circle (0.5);
    \draw [fill=yellow, opacity=0.25] ( \x+3,3) circle (0.5);
    }
    
    \draw[thin, orange] ( 1,1)--(11,1);
    \draw[thin, orange] ( 1,3)--(11,3);

    \foreach \x in {0, 2, 4, 6, 8, 10} {
        \node[blue] at (\x+1, 1) {$\bullet$};
        \node[blue] at (\x+1, 3) {$\bullet$};
        }

    \node at ( 2, 1.25)      {$e_{1\,1}$};
    \node at ( 4, 1.25)      {$e_{1\,2}$};
    \node at ( 6, 1.25)      {$e_{1\,3}$};
    \node at ( 8, 1.25)      {$e_{1\,4}$};
    \node at (10, 1.25)      {$e_{1\,5}$};

    \node at ( 2, 3.25)      {$e_{2\,5}$};
    \node at ( 4, 3.25)      {$e_{2\,4}$};
    \node at ( 6, 3.25)      {$e_{2\,3}$};
    \node at ( 8, 3.25)      {$e_{2\,2}$};
    \node at (10, 3.25)      {$e_{2\,1}$};

    \end{tikzpicture}
    \caption{We show which points are identified once the cut and paste construction is performed. }
    \label{fig:hypcasebase3}
\end{figure}
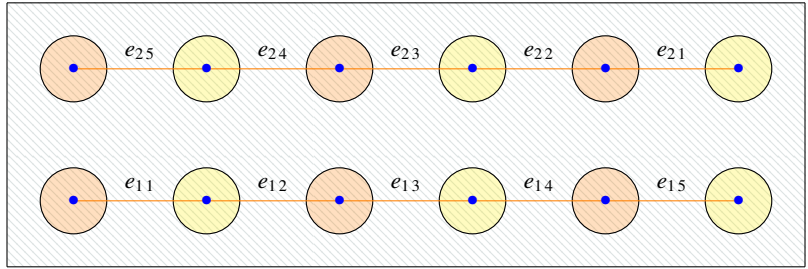

\noindent By construction, there is a clear symmetry obtained by rotating the slit rectangle $\mathcal R$ by the angle $\pi$ around some point, say $c$, in its interior (such a point is the center of $\mathcal R$ if and only if $q_{i\,j}$ is the center of the square to which it belongs). Such a symmetry determines an involution with $2g+2$ fixed points listed below: 
\begin{itemize}
    \item[1.] Four out of $2g+2$ of these points come from the original base rectangle $\mathcal R$ and these are the vertices of some inner rectangle of volume $g/2$. 
    \smallskip
    \item[2.] By rotating $\mathcal R$ around $c$, the mid-point of $e_{1\,j}$ is mapped to that of $e_{2\,g-j}$. In particular, the mid-point of $e_{1\,j}^{\pm}$ is mapped to that of $e_{2\,g-j}^{\mp}$. Once these edges are identified as described above, the resulting structure naturally has $2g-2$ additional fixed points. 
    \end{itemize}
Thus the resulting structure is hyperelliptic as desired; see Figure \ref{fig:hypcasebase4}. We observe, finally, that the construction just described relies on the fact that the base rectangle has an even number of square tiles, where the minimum required degree is \(2g\). If we did not require the relative periods to be integral, then the minimum degree needed to realize such a structure (with the same signature and hyperelliptic structure) would be halved, \textit{i.e.}, \(g\). In \cite[Section \S3.2.1]{BJJP}, the authors provided a different construction that could also have been applied to our case. However, in order to keep our work self-contained, we prefer to provide our own explicit construction.

\begin{figure}[h!]
    \centering
    \begin{tikzpicture}[scale=0.875, every node/.style={scale=0.875}]
    \definecolor{pallido}{RGB}{221,227,227}

    \pattern [pattern=north west lines, pattern color=pallido]
    (0,0)--(12,0)--(12,4)--(0,4)--(0,0);
    
    \draw[thin, black] (0,0)--(12,0);
    \draw[thin, black] (0,4)--(12,4);
    \foreach \x in {0, 12} {
        \draw (\x, 0) -- (\x, 4);
        }

    \draw[thin, black] ( 1.5,1.5)--(11.5,1.5);
    \draw[thin, black] ( 1.5,3.5)--(11.5,3.5);

    \draw[thin, dotted, black] ( 0.5,0.5)--(6.5,0.5);
    \draw[thin, dotted, black] ( 0.5,2.5)--(6.5,2.5);
    \draw[thin, dotted, black] ( 0.5,0.5)--(0.5,2.5);
    \draw[thin, dotted, black] ( 6.5,0.5)--(6.5,2.5);

    \foreach \x in {0, 2, 4, 6, 8, 10} {
        \node[blue] at (\x+1.5, 1.5) {$\bullet$};
        \node[blue] at (\x+1.5, 3.5) {$\bullet$};
        }

    \node[violet] at ( 0.5, 2.5)   {$\times$};
    \node[violet] at ( 6.5, 2.5)   {$\times$};
    \node[violet] at ( 0.5, 0.5)   {$\times$};
    \node[violet] at ( 6.5, 0.5)   {$\times$};
    
    \node[red] at ( 2.5, 1.5)      {$\times$};
    \node[red] at ( 4.5, 1.5)      {$\times$};
    \node[red] at ( 6.5, 1.5)      {$\times$};
    \node[red] at ( 8.5, 1.5)      {$\times$};
    \node[red] at (10.5, 1.5)      {$\times$};

    \node[red] at ( 2.5, 3.5)      {$\times$};
    \node[red] at ( 4.5, 3.5)      {$\times$};
    \node[red] at ( 6.5, 3.5)      {$\times$};
    \node[red] at ( 8.5, 3.5)      {$\times$};
    \node[red] at (10.5, 3.5)      {$\times$};

    \node at (11.5,0.5)        {$\mathcal R$};

    \end{tikzpicture}
    \caption{Fixed points of the hyperelliptic involution. The labels of points and edges have been removed. For those, the reader can refer to Figure \ref{fig:hypcasebase2}. For $g=6$, there are $2g+2=14$ fixed points. In violet those fixed points come from the base rectangle $\mathcal R$. The dotted lines show the inner rectangle of volume $g/2$. In red we denote the midpoints of the edges $e_{i\,j}$.}
    \label{fig:hypcasebase4}
\end{figure}
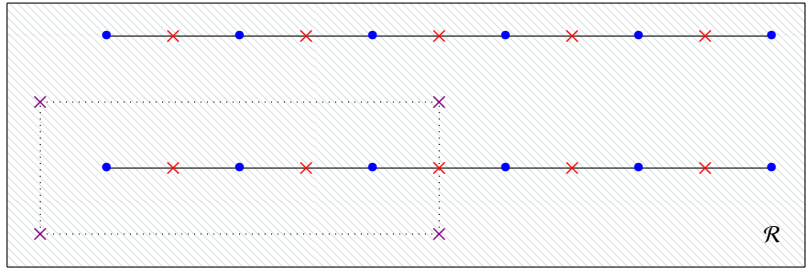

\subsubsection{The non-hyperelliptic component}\label{sssec:nothypcomp} The representation $\chi$ can be also realized in the non-hyperelliptic component of $\mathcal{H}(g-1,g-1)$. 
Let $\mathcal R$ be a rectangle of size $g\,\times\,2$. By gluing the opposite sides via translation, we obtain a flat torus $T$ of volume $2g$ that comes with a genuine covering map $\pi\colon T\longrightarrow \mathbb T$ of degree $2g$ by construction. Let $q\in\mathbb T$ be any point and let $Q$ denote the set $\{\,q_{1\,1},\dots,q_{1\,2g},q_{2\,1},\dots,q_{2\,2g}\,\}$ of preimages of $q$ via the covering $\pi$ and their labels are as in \S\ref{sssec:hypercomp}. In this case, we join $q_{1\,j}$ and $q_{2\,j}$ with an edge, say $e_j$, as shown in Figure \ref{fig:hypcasebase5}.

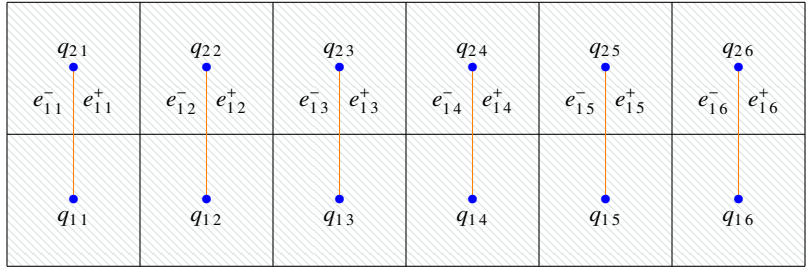
\begin{figure}[h!]
    \centering
    \begin{tikzpicture}[scale=0.875, every node/.style={scale=0.875}]
    \definecolor{pallido}{RGB}{221,227,227}

    \pattern [pattern=north west lines, pattern color=pallido]
    (0,0)--(12,0)--(12,4)--(0,4)--(0,0);
    
    \draw[thin, black] (0,0)--(12,0);
    \draw[thin, black] (0,2)--(12,2);
    \draw[thin, black] (0,4)--(12,4);
    \foreach \x in {0, 2, 4, 6, 8, 10, 12} {
        \draw (\x, 0) -- (\x, 4);
        }

    \foreach \x in {0, 2, 4, 6, 8, 10} {
        \draw[thin, orange] ( \x+1,1)--(\x+1,3);
        }

    \foreach \x in {0, 2, 4, 6, 8, 10} {
        \node[blue] at (\x+1, 1) {$\bullet$};
        \node[blue] at (\x+1, 3) {$\bullet$};
        }


    \node at ( 1.375, 2.5)      {$e_{1\,1}^+$};
    \node at ( 3.375, 2.5)      {$e_{1\,2}^+$};
    \node at ( 5.375, 2.5)      {$e_{1\,3}^+$};
    \node at ( 7.375, 2.5)      {$e_{1\,4}^+$};
    \node at ( 9.375, 2.5)      {$e_{1\,5}^+$};
    \node at (11.375, 2.5)      {$e_{1\,6}^+$};

    \node at ( 0.625, 2.5)      {$e_{1\,1}^-$};
    \node at ( 2.625, 2.5)      {$e_{1\,2}^-$};
    \node at ( 4.625, 2.5)      {$e_{1\,3}^-$};
    \node at ( 6.625, 2.5)      {$e_{1\,4}^-$};
    \node at ( 8.625, 2.5)      {$e_{1\,5}^-$};
    \node at (10.625, 2.5)      {$e_{1\,6}^-$};

    \node at ( 1,0.725)      {$q_{1\,1}$};
    \node at ( 3,0.725)      {$q_{1\,2}$};
    \node at ( 5,0.725)      {$q_{1\,3}$};
    \node at ( 7,0.725)      {$q_{1\,4}$};
    \node at ( 9,0.725)      {$q_{1\,5}$};
    \node at (11,0.725)      {$q_{1\,6}$};

    \node at ( 1,3.275)      {$q_{2\,1}$};
    \node at ( 3,3.275)      {$q_{2\,2}$};
    \node at ( 5,3.275)      {$q_{2\,3}$};
    \node at ( 7,3.275)      {$q_{2\,4}$};
    \node at ( 9,3.275)      {$q_{2\,5}$};
    \node at (11,3.275)      {$q_{2\,6}$};
    
    \end{tikzpicture}
    \caption{Realization in the non-hyperelliptic component.}
    \label{fig:hypcasebase5}
\end{figure}

\noindent We then slit all edges $e_j$ and denote the resulting sides with $e_j^{\pm}$ according to our convention --  here we can assume without loss of generality that all edges are oriented upwards. Then, we  glue $e_j^+$ with $e_{j+1}^-$ and $e_g^+$ with $e_1^-$. 

\smallskip

\noindent The resulting structure, say $(X,\omega)$, is a translation surface with signature \((\,g-1,\,g-1\,)\) and hence, it has genus \(g\). It remains to show that such a structure is not hyperelliptic for every \(g\ge3\). We argue by contradiction and suppose there exists a hyperelliptic involution $\tau$. By construction, there are \(g\) vertical strips, say $S_i$, each of length \(2\). Notice that each such a strip corresponds to a slit torus. Thus \((X,\omega)\) can be seen as the structure obtained from a collection of \(g\) slit torus glued along slits, say \(s_1,\dots,s_g\), of unit length; see Figure \ref{fig:hypcasebase6}. In principle, two vertical strips may be swapped by \(\tau\). We claim that this cannot be the case. In fact, if \(\tau\) were to map a vertical strip \(S_i\) onto another, distinct, vertical strip \(S_j\), then \(X/\tau\) would contain \((S_i\,\cup\,S_j)/\tau\), which is an embedded one-holed torus, and hence the quotient cannot be a sphere. As a consequence, \(\tau\) must preserve all slit tori \(S_1,\dots,S_g\), and hence it must preserve all saddle connections \(s_i\) in the sense that \(\tau(\,s_i\,)=s_i\). Therefore, the mid-point of each saddle connection must be a fixed point. Notice that there are \(g\) of such points. The desired contradiction directly follows, because all of these points belong to the unique horizontal cylinder, say \(C\), of length \(2g\) which must be fixed by \(\tau\), and hence it must contain only two fixed points.

\begin{figure}[h!]
    \centering
    \begin{tikzpicture}[scale=0.875, every node/.style={scale=0.875}]
    \definecolor{pallido}{RGB}{221,227,227}

    \pattern [pattern=north west lines, pattern color=pallido]
    (0,0)--(12,0)--(12,4)--(0,4)--(0,0);
    
    \draw[thin, black] (0,0)--(12,0);
    \draw[thin, dashed, black] (0,1)--(12,1);
    \draw[thin, dashed, black] (0,3)--(12,3);
    \draw[thin, black] (0,4)--(12,4);
    \foreach \x in {0,  12} {
        \draw (\x, 0) -- (\x, 4);
        }

    \foreach \x in {0,2,4,6,8,10} {

    \draw[thin, dashed, black] (\x+1,3)--(\x+1,4);
    \draw[thin, dashed, black] (\x+1,0)--(\x+1,1);
    }

    \draw [fill=blue, opacity=0.125] (3,0)--(5,0)--(5,4)--(3,4)--(3,0);

    \draw [fill=orange, opacity=0.25] (0,0)--(12,0)--(12,1)--(0,1)--(0,0);
    \draw [fill=orange, opacity=0.25] (0,3)--(12,3)--(12,4)--(0,4)--(0,3);

    \foreach \x in {0, 2, 4, 6, 8, 10} {
        \draw[thin, orange] ( \x+1,1)--(\x+1,3);
        }

    \foreach \x in {0, 2, 4, 6, 8, 10} {
        \node[blue] at (\x+1, 1) {$\bullet$};
        \node[blue] at (\x+1, 3) {$\bullet$};
        }

    \node at (11.5,0.5) {$C$};
    \node at ( 4  ,0.5) {$S_3$};

    \node at ( 0.625, 3.5)      {$s_1$};
    \node at ( 2.625, 3.5)      {$s_2$};
    \node at ( 4.625, 3.5)      {$s_3$};
    \node at ( 6.625, 3.5)      {$s_4$};
    \node at ( 8.625, 3.5)      {$s_5$};
    \node at (10.625, 3.5)      {$s_6$};

    \end{tikzpicture}
    \caption{In orange is depicted the cylinder of length $2g$. The interior of its complement is a collection of $g$ cylinders of unit length. The vertical strip in blue shows a slit torus whose slits correspond to the dashed lines.}
    \label{fig:hypcasebase6}
\end{figure}
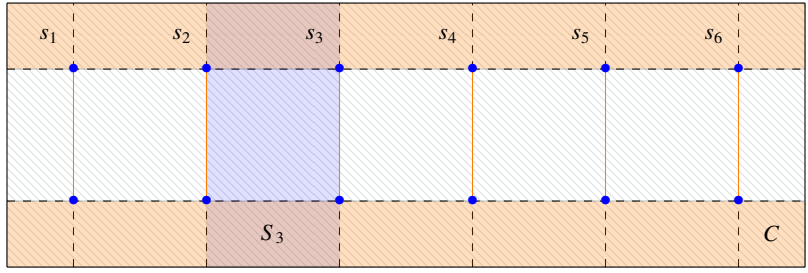

\subsubsection{Higher degree}\label{sssec:hdhyp} All square-tiled surfaces above come with a branched covering $\pi\colon X\longrightarrow \mathbb T$ of minimal degree $2g$. In order to produce covering maps of degree $d>2g$, we can modify slightly our previous constructions to add additional $d-2g$ unit squares. In both of the constructions mentioned above, it is possible to find simple closed geodesics such that, by cutting the surface along them, the resulting surface has two boundary components to which we can attach a cylinder of arbitrary length. 

\begin{figure}[h!]
    \centering
    \begin{tikzpicture}[scale=0.85, every node/.style={scale=0.75}]
    \definecolor{pallido}{RGB}{221,227,227}

    \pattern [pattern=north west lines, pattern color=pallido]
    (0,0)--(12,0)--(12,4)--(0,4)--(0,0);

    \pattern [pattern=north west lines, pattern color=pallido]
    (1,-1)--(5.75,-1)--(5.75,-3)--(1,-3)--(1,-1);

    \pattern [pattern=north west lines, pattern color=pallido]
    (11,-1)--(8.25,-1)--(8.25,-3)--(11,-3)--(11,-1);

    \draw[thin, black] (0,0)--(12,0);
    \draw[thin, black] (0,4)--(12,4);

    \draw[thin, black] (1,-1)--(5.25,-1);
    \draw[thin, black, dashed] (5.25,-1)--(5.75,-1);
    \draw[thin, black] (1,-3)--(5.25,-3);
    \draw[thin, black, dashed] (5.25,-3)--(5.75,-3);

    \draw[thin, black] (8.75,-1)--(11,-1);
    \draw[thin, black, dashed] (8.25,-1)--(8.75,-1);
    \draw[thin, black] (8.75,-3)--(11,-3);
    \draw[thin, black, dashed] (8.25,-3)--(8.75,-3);

    \foreach \x in {0,12} {
        \draw[thin, black] (\x, 0) -- (\x, 4);
        }

    \draw[thin, violet] (1, -1) -- (1, -3);
    \foreach \x in {4,6, 10} {
        \draw[thin, black] (\x-1, -1) -- (\x-1, -3);
        }
    \draw[thin, violet] (11, -1) -- (11, -3);
    
    \draw[thin, violet] (6,1)--(6,3);

    \draw[thin, orange] ( 1,1)--(11,1);
    \draw[thin, orange] ( 1,3)--(11,3);

    \foreach \x in {0, 2, 4, 6, 8, 10} {
        \node[blue] at (\x+1, 1) {$\bullet$};
        \node[blue] at (\x+1, 3) {$\bullet$};
        }

    \node at (11.5,0.5)      {$\mathcal R$};

    \node at ( 2, 0.75)      {$e_{1\,1}$};
    \node at ( 4, 0.75)      {$e_{1\,2}$};
    \node at ( 6, 0.75)      {$e_{1\,3}$};
    \node at ( 8, 0.75)      {$e_{1\,4}$};
    \node at (10, 0.75)      {$e_{1\,5}$};

    \node at ( 2, 3.25)      {$e_{2\,5}$};
    \node at ( 4, 3.25)      {$e_{2\,4}$};
    \node at ( 6, 3.25)      {$e_{2\,3}$};
    \node at ( 8, 3.25)      {$e_{2\,2}$};
    \node at (10, 3.25)      {$e_{2\,1}$};

    \node at (7,-2) {$\cdots\,\cdots$};

    \node at ( 0.75,-2) {$t^+$};
    \node at (11.25,-2) {$t^-$};
    \node at ( 5.75, 2) {$s$};
    \node at ( 2,   -2) {$2g+1$};
    \node at ( 4,   -2) {$2g+2$};
    \node at (10,   -2) {$d$};
    
    \end{tikzpicture}
    \caption{Realization of a hyperelliptic translation surface with signature \((g-1, g-1)\) and prescribed absolute and relative periods with \(g\) even. We slit along the simple closed curve \(s\), drawn in violet, and we call the resulting sides $s^{\pm}$ according to our convention. Then, we glue the strip \(S\) by identifying the horizontal sides via translation and then glue $s^{\pm}$ with $t^{\mp}$. The resulting structure is the desired one. A similar construction works for odd values of \(g\) as well.} 
    \label{fig:hypcasebase7}
\end{figure}
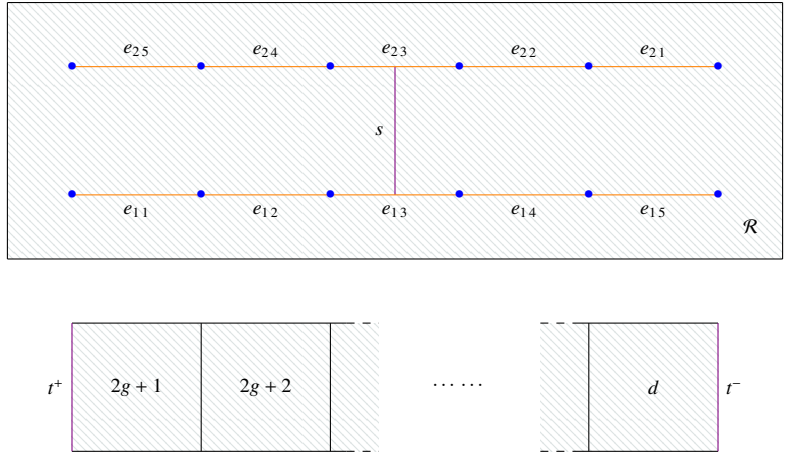

\noindent If the cylinder consists of \(d-2g\) unit squares, the desired surface is a square-tiled surface and comes with a branched covering of degree \(d\) over the standard torus. We provide two more pictures to demonstrate the idea without going through the details; see Figures \ref{fig:hypcasebase7} and \ref{fig:hypcasebase8}. 

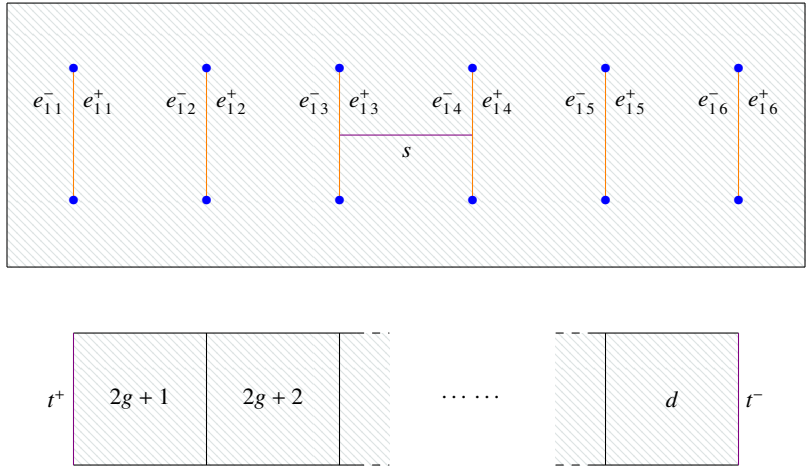
\begin{figure}[h!]
    \centering
    \begin{tikzpicture}[scale=0.875, every node/.style={scale=0.875}]
    \definecolor{pallido}{RGB}{221,227,227}

    \pattern [pattern=north west lines, pattern color=pallido]
    (0,0)--(12,0)--(12,4)--(0,4)--(0,0);

    \pattern [pattern=north west lines, pattern color=pallido]
    (1,-1)--(5.75,-1)--(5.75,-3)--(1,-3)--(1,-1);

    \pattern [pattern=north west lines, pattern color=pallido]
    (11,-1)--(8.25,-1)--(8.25,-3)--(11,-3)--(11,-1);
    
    \draw[thin, black] (0,0)--(12,0);
    \draw[thin, black] (0,4)--(12,4);
    \foreach \x in {0,12} {
        \draw (\x, 0) -- (\x, 4);
        }

    \draw[thin, black] (1,-1)--(5.25,-1);
    \draw[thin, black, dashed] (5.25,-1)--(5.75,-1);
    \draw[thin, black] (1,-3)--(5.25,-3);
    \draw[thin, black, dashed] (5.25,-3)--(5.75,-3);

    \draw[thin, black] (8.75,-1)--(11,-1);
    \draw[thin, black, dashed] (8.25,-1)--(8.75,-1);
    \draw[thin, black] (8.75,-3)--(11,-3);
    \draw[thin, black, dashed] (8.25,-3)--(8.75,-3);

    \draw[thin, violet] (1, -1) -- (1, -3);
    \foreach \x in {4,6, 10} {
        \draw[thin, black] (\x-1, -1) -- (\x-1, -3);
        }
    \draw[thin, violet] (11, -1) -- (11, -3);
    
    \draw[thin, violet] (5,2)--(7,2);
    
    \foreach \x in {0,2,4,6,8, 10} {
        \draw[thin, orange] ( \x+1,1)--(\x+1,3);
        }

    \foreach \x in {0, 2, 4, 6, 8, 10} {
        \node[blue] at (\x+1, 1) {$\bullet$};
        \node[blue] at (\x+1, 3) {$\bullet$};
        }

    \node at (7,-2) {$\cdots\,\cdots$};
      
    \node at ( 1.375, 2.5)      {$e_{1\,1}^+$};
    \node at ( 3.375, 2.5)      {$e_{1\,2}^+$};
    \node at ( 5.375, 2.5)      {$e_{1\,3}^+$};
    \node at ( 7.375, 2.5)      {$e_{1\,4}^+$};
    \node at ( 9.375, 2.5)      {$e_{1\,5}^+$};
    \node at (11.375, 2.5)      {$e_{1\,6}^+$};

    \node at ( 0.625, 2.5)      {$e_{1\,1}^-$};
    \node at ( 2.625, 2.5)      {$e_{1\,2}^-$};
    \node at ( 4.625, 2.5)      {$e_{1\,3}^-$};
    \node at ( 6.625, 2.5)      {$e_{1\,4}^-$};
    \node at ( 8.625, 2.5)      {$e_{1\,5}^-$};
    \node at (10.625, 2.5)      {$e_{1\,6}^-$};

    \node at ( 0.75,-2) {$t^+$};
    \node at (11.25,-2) {$t^-$};
    \node at ( 6, 1.75) {$s$};
    \node at ( 2,   -2) {$2g+1$};
    \node at ( 4,   -2) {$2g+2$};
    \node at (10,   -2) {$d$};
    
    \end{tikzpicture}
    \caption{Realization of a non-hyperelliptic translation surface with $d>2g$ unit squares. The additional \(d-2g\) squares are drawn horizontally to simplify the picture. The same idea as described in the caption of Figure \ref{fig:hypcasebase7} can be applied here \textit{mutatis mutandis}.}
    \label{fig:hypcasebase8}
\end{figure}


\medskip

\subsection{Realization in the even and odd components of strata}\label{ssec:lastcasespintwo}


\noindent We finally develop a construction to realize square-tiled surfaces as in Definition \ref{defn:transurf} with prescribed branching data and parity. 

\smallskip

\noindent Let $\mu=(\,\mu_o,\dots,\mu_\ell\,)$ be a partitioned signature. At least one of these sub-signatures realizes ${\rm m}(\,\chi\,,\,\mu^\vdash\,)
$, that is the lowest possible degree for a branched covering $\pi\colon X\longrightarrow \mathbb T$ such that $\omega=\pi^*dz$ has signature $\mu$ and branching data prescribed according to $\mu^\vdash$. As a consequence,  $$d=\deg(\,\pi\,)\ge\,|\,\mu_i\,|+|\,Z_i\,|$$ 
for all $i=0,\dots,\ell$. Let $\nu=(2m_1,\dots,2m_h)$ be any one of $\mu_o,\dots,\mu_\ell$. Then, the inequality above can be written as 
\begin{equation}\label{eq:subsigncond}
    \sum_{i=1}^h (2m_i+1) \le d\quad \text{ that is } \quad \sum_{i=1}^h 2m_i \,\le\,d-h.
\end{equation}
Thus, there are finitely many cases for an even signature $\nu$ to appear as a sub-signature of $\mu$. We further notice that, since $2m_i\ge2$ for all $i=1,\dots,h$, it follows that $3h\le d$, which provides an upper bound on the possible lengths allowed for a sub-partition.

\begin{ex}
    As an example, we show the possible signatures in the case of $d=11$ in Table~\ref{tab:subsign}. Notice that $h\le3$ in this case.
    \begin{table}[h!]
        \centering
        \begin{tabular}{ccc}
        \toprule
            $h=1$  & $h=2$ & $h=3$\\
        \midrule
            $(\,2\,)$  & $(\,2\,\,2\,)$ & $(\,2\,\,2\,\,2\,)$\\
            $(\,4\,)$  & $(\,2\,\,4\,)$ & $(\,2\,\,2\,\,4\,)$\\
            $(\,6\,)$  & $(\,2\,\,6\,)$ & \\
            $(\,8\,)$  & $(\,4\,\,4\,)$ & \\
            $(\,10\,)$ &          & \\
        \bottomrule
        \end{tabular}
        \caption{Possible sub-signatures with $d=11$.}
        \label{tab:subsign}
    \end{table}
\end{ex}

\smallskip

\begin{rmk}
    If we relax the assumption on the minimal degree in Proposition \ref{prop:branchcoverings}, \textit{i.e.}, if we do not impose any conditions on the number of squares required to realize the desired structure, the final case can be easily achieved by constructing an origami in the given stratum. Then, by applying the zero-movement technique described in Section \S\ref{sssec:schiffer} and using the same approach developed in Section \S\ref{ssec:nointrelper}, we can achieve the desired result. In the following, we present an alternative argument to cover the missing cases. Since this argument is applicable for all degrees, we provide the general description.
\end{rmk}

\noindent The gist of the idea for partitioned signatures is as follows. We first realize an origami $(X,\omega)$, and then we realize the desired sub-signatures by bubbling handles by means of slits. The choice of $(X, \omega)$ depends on the parity that we aim to realize. More precisely, we distinguish two possible cases ${\rm O}$ and ${\rm E}$ as follows:
\begin{itemize}
    \item[O] we choose $(X,\omega)$ as the torus obtained by lining up $d$ unit squares, if the desired parity is odd. Notice that a torus always has odd spin parity. Otherwise, 
    \smallskip
    \item[E] we choose $(X, \omega)$ as an origami of genus three with even parity. In this case, there are two possible strata with even parity: $\mathcal{H}_3(2,2)$ and $\mathcal{H}_3(\,4\,)$. In what follows, we primarily focus on origamis from the former stratum. 
\end{itemize}

\noindent In the following subsections, we handle these cases separately.

\subsubsection{Case ${\rm O}$.}\label{sssec:oddcase} This is an easier case to handle, as no exceptional or special cases arise, allowing us to use a uniform argument. Let $\mu=(\mu_o,\dots,\mu_\ell)$ be a signature and let $d={\rm m}(\,\chi\,,\,\mu^\vdash\,)$. Let $\mathcal R$ be the rectangle obtained by lining up $d$ unit squares. By identifying the opposite sides as usual by means of translation, we obtain a translation surface, say $(X,\omega)$, of volume $d$ homeomorphic to a torus. There is a natural covering map $\pi\colon X\longrightarrow \mathbb T$ of degree $d$. Keeping Remark \ref{rmk:relperint} in mind, recall that we can regard relative periods as elements of $\mathbb{T}$. Let $\mu$ be partitioned as $(\mu_0, \dots, \mu_\ell)$, and let $r_1, \dots, r_\ell \in \mathbb{T}$ be the desired relative periods. Upon choosing a base point $y_o \in \mathbb{T}$, define $y_i = y_o + r_i$. Note that each of these points, being a regular value of $\pi$, has exactly $d$ preimages — one for each square of $(X, \omega)$. 

\smallskip

\noindent Let $Q_o=\{\,q_{o1},\dots,q_{od}\,\}$ be the preimages of $y_o$. This forms a collection of marked points on $(X, \omega)$, and by connecting them appropriately with segments — where the way we join these segments is encoded in the sub-partition — we can perform a slit construction to bubble a handle; see Figure \ref{fig:oddspintransurf}.

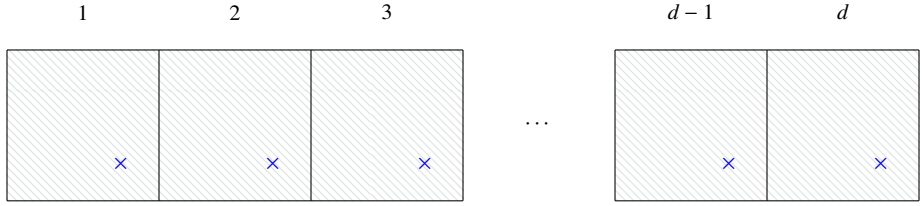
\begin{figure}[h!]
    \centering
    \begin{tikzpicture}[scale=1, every node/.style={scale=0.825}]
    \definecolor{pallido}{RGB}{221,227,227}

    \pattern [pattern=north west lines, pattern color=pallido]
    (0,0)--(6,0)--(6,2)--(0,2)--(0,0);

    \pattern [pattern=north west lines, pattern color=pallido]
    (8,0)--(12,0)--(12,2)--(8,2)--(8,0);
    
    \draw[thin, black] (0,0)--(6,0);
    \draw[thin, black] (0,2)--(6,2);
    \foreach \x in {0, 2, 4, 6} {
        \draw (\x, 0) -- (\x, 2);
        }

    \draw[thin, black] (8,0)--(12,0);
    \draw[thin, black] (8,2)--(12,2);
    \foreach \x in {8,10,12} {
        \draw (\x, 0) -- (\x, 2);
        }

    \node at ( 7,  1) {$\cdots$};
    \node at ( 1,2.5)      {$1$};
    \node at ( 3,2.5)      {$2$};
    \node at ( 5,2.5)      {$3$};
    \node at ( 9,2.5)    {$d-1$};
    \node at (11,2.5)      {$d$};

    \foreach \x in {0, 2, 4, 8, 10} {
        \node[blue] at (\x+1.5, 0.5) {$\times$};
        }
    
    \end{tikzpicture}
    \caption{Lined-up squares with marked points}
    \label{fig:oddspintransurf}
\end{figure}

\noindent Let $\mu_o=(\,2m_1,\dots,2m_{h_o}\,)$. We apply the following rule: Two points can be joined with an edge if they are adjacent in the sense that they belong to two adjacent squares. Notice that, as an artifact of our construction, the squares labeled with $1$ and $d$ are considered as adjacent -- this will be relevant in Section \S\ref{sssec:evencase} below. According to this rule, we join $q_{oi}$ and $q_{oi+1}$ with an edge labeled as $e_i$ for $i=1,\dots,2m_1$ and orient all edges from the left to right. Slit all edges and label the resulting sides according to our convention and then identify $e_{2i-1}^{\pm}$ with $e_{2i}^{\mp}$ for all $i=1,\dots,m_1$. The resulting structure is a square-tiled surface 
and, by construction, the corresponding abelian differential $\xi$ has a single zero of order $2m_1$ arising from  collapsing the points $q_{o1},\dots,q_{0,2m_1+1}$ of $(X,\omega)$. 

\smallskip

\noindent Next, we iterate the same construction: we do \textit{not} join $q_{0,2m_1+1}$ with $q_{0,2m_1+2}$ and we keep joining adjacent points for $i=2m_1+2,\dots, 2m_1+2m_2+2$. The essential idea relies on the ability to realize chains of edges consisting of $2m_i$ edges with at most $d-h_o$ gaps. Notice that the necessary condition \eqref{eq:subsigncond} ensures that this is always possible. In other words, every signature that satisfies \eqref{eq:subsigncond} can be realized; see Figure \ref{fig:oddspintransurf2}. If $\ell=0$, \textit{i.e.} $\mu$ is not partitioned, then the construction is complete, and the resulting translation surface is in fact an origami. It remains to verify that it has the desired parity, which we will address later. For now, let us focus on the case of partitioned signatures.

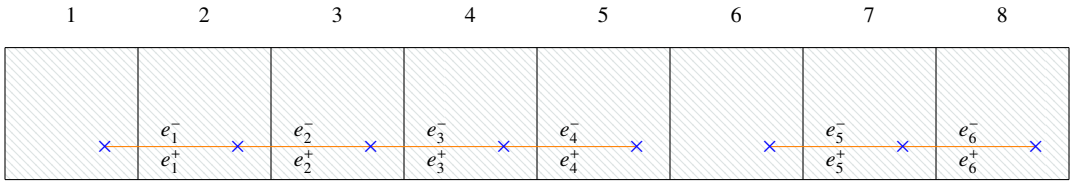
\begin{figure}[h!]
    \centering
    \begin{tikzpicture}[scale=0.875, every node/.style={scale=0.825}]
    \definecolor{pallido}{RGB}{221,227,227}

    \pattern [pattern=north west lines, pattern color=pallido]
    (0,0)--(16,0)--(16,2)--(0,2)--(0,0);
    
    \draw[thin, black] (0,0)--(16,0);
    \draw[thin, black] (0,2)--(16,2);
    \foreach \x in {0, 2, 4, 6, 8, 10, 12, 14, 16} {
        \draw (\x, 0) -- (\x, 2);
        }

    \node at ( 1,2.5)      {$1$};
    \node at ( 3,2.5)      {$2$};
    \node at ( 5,2.5)      {$3$};
    \node at ( 7,2.5)      {$4$};
    \node at ( 9,2.5)      {$5$};
    \node at (11,2.5)      {$6$};
    \node at (13,2.5)      {$7$};
    \node at (15,2.5)      {$8$};

    \draw[thin, orange] ( 1.5,0.5)--( 3.5,0.5);
    \draw[thin, orange] ( 3.5,0.5)--( 5.5,0.5);
    \draw[thin, orange] ( 5.5,0.5)--( 7.5,0.5);
    \draw[thin, orange] ( 7.5,0.5)--( 9.5,0.5);

    \draw[thin, orange] (11.5,0.5)--(13.5,0.5);
    \draw[thin, orange] (13.5,0.5)--(15.5,0.5);
    
    \foreach \x in {0, 2, 4, 6, 8, 10, 12, 14} {
        \node[blue] at (\x+1.5, 0.5) {$\times$};
        }

    \node at ( 2.5,0.725)      {$e_1^-$};
    \node at ( 2.5,0.275)      {$e_1^+$};
    \node at ( 4.5,0.725)      {$e_2^-$};
    \node at ( 4.5,0.275)      {$e_2^+$};
    \node at ( 6.5,0.725)      {$e_3^-$};
    \node at ( 6.5,0.275)      {$e_3^+$};
    \node at ( 8.5,0.725)      {$e_4^-$};
    \node at ( 8.5,0.275)      {$e_4^+$};
    \node at (12.5,0.725)      {$e_5^-$};
    \node at (12.5,0.275)      {$e_5^+$};
    \node at (14.5,0.725)      {$e_6^-$};
    \node at (14.5,0.275)      {$e_6^+$};

    \end{tikzpicture}
    \caption{Realization of a given sub-signature. In this picture is depicted how to realize the signature $(\,4,2\,)$. Notice that $d=8$ is the minimal degree that allows the realization of such a signature.}
    \label{fig:oddspintransurf2}
\end{figure}

\noindent We now assume that $\mu$ is a partitioned signature and we extend the process above to realize square-tiled surfaces. 
Let $y_1$ be the point such that $y_1=y_o+r_1$, where $r_1$ is the desired relative period. It has $d$ preimages via the mapping $\pi$, say $Q_1=\{\,q_{11},\dots,q_{1d}\,\}$, and we notice that $Q_o$ and $Q_1$ are disjoint subsets because $r_1\neq0$. Let $\mu_1=(2m_{h_o+1},\dots,2m_{h_o+h_1})$ be a signature of length $h_1$ such that \eqref{eq:subsigncond} is satisfied. Then, we can perform the same construction as above. In other words, we can find $h_1$ chains consisting of $2m_i$ edges with at most $d-h_1$ gaps, where $i=2m_{h_o+1},\dots,2m_{h_o+h_1}$; see Figure \ref{fig:oddspintransurf3}. By slitting and re-gluing the edges as shown above, we obtain the desired surface.

\begin{figure}[h!]
    \centering
    \begin{tikzpicture}[scale=0.875, every node/.style={scale=0.825}]
    \definecolor{pallido}{RGB}{221,227,227}

    \pattern [pattern=north west lines, pattern color=pallido]
    (0,0)--(16,0)--(16,2)--(0,2)--(0,0);
    
    \draw[thin, black] (0,0)--(16,0);
    \draw[thin, black] (0,2)--(16,2);
    \foreach \x in {0, 2, 4, 6, 8, 10, 12, 14, 16} {
        \draw (\x, 0) -- (\x, 2);
        }

    \node at ( 1,2.5)      {$1$};
    \node at ( 3,2.5)      {$2$};
    \node at ( 5,2.5)      {$3$};
    \node at ( 7,2.5)      {$4$};
    \node at ( 9,2.5)      {$5$};
    \node at (11,2.5)      {$6$};
    \node at (13,2.5)      {$7$};
    \node at (15,2.5)      {$8$};

    \draw[thin, orange] ( 1.5,0.5)--( 3.5,0.5);
    \draw[thin, orange] ( 3.5,0.5)--( 5.5,0.5);
    \draw[thin, orange] ( 5.5,0.5)--( 7.5,0.5);
    \draw[thin, orange] ( 7.5,0.5)--( 9.5,0.5);

    \draw[thin, orange] (11.5,0.5)--(13.5,0.5);
    \draw[thin, orange] (13.5,0.5)--(15.5,0.5);
    
    \draw[thin, red] ( 1,1.5)--( 3,1.5);
    \draw[thin, red] ( 3,1.5)--( 5,1.5);

    \draw[thin, red] ( 9,1.5)--(11,1.5);
    \draw[thin, red] (11,1.5)--(13,1.5);

    \foreach \x in {0, 2, 4, 6, 8, 10, 12, 14} {
        \node[blue] at (\x+1.5, 0.5) {$\times$};
        }

    \foreach \x in {0, 2, 4, 8, 10, 12} {
        \node[blue] at (\x+1, 1.5) {$\times$};
        }
    
    \end{tikzpicture}
    \caption{Realization of a given sub-signature. In this picture is depicted how to realize the signature $\mu=(\,\mu_o,\mu_1\,)$, where $\mu_o=(\,4,2\,)$ and $\mu_1=(\,2,2\,)$.}
    \label{fig:oddspintransurf3}
\end{figure}
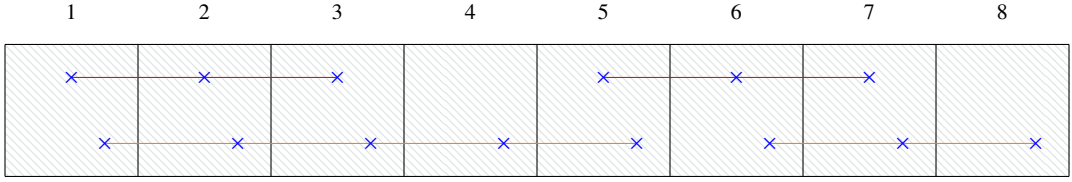

\noindent By using an inductive argument on $\ell$, it can be shown that every partitioned signature $\mu=(\mu_o,\dots,\mu_\ell)$ can be realized in this manner in $\ell$ steps, as illustrated in Figure \ref{fig:oddspintransurf4}.

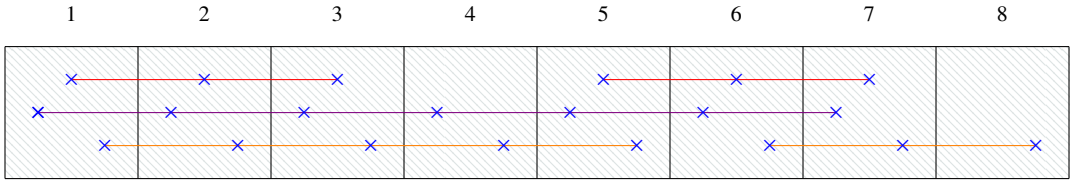
\begin{figure}[h!]
    \centering
    \begin{tikzpicture}[scale=0.875, every node/.style={scale=0.825}]
    \definecolor{pallido}{RGB}{221,227,227}

    \pattern [pattern=north west lines, pattern color=pallido]
    (0,0)--(16,0)--(16,2)--(0,2)--(0,0);
    
    \draw[thin, black] (0,0)--(16,0);
    \draw[thin, black] (0,2)--(16,2);
    \foreach \x in {0, 2, 4, 6, 8, 10, 12, 14, 16} {
        \draw (\x, 0) -- (\x, 2);
        }

    \node at ( 1,2.5)      {$1$};
    \node at ( 3,2.5)      {$2$};
    \node at ( 5,2.5)      {$3$};
    \node at ( 7,2.5)      {$4$};
    \node at ( 9,2.5)      {$5$};
    \node at (11,2.5)      {$6$};
    \node at (13,2.5)      {$7$};
    \node at (15,2.5)      {$8$};

    \draw[thin, orange] ( 1.5,0.5)--( 3.5,0.5);
    \draw[thin, orange] ( 3.5,0.5)--( 5.5,0.5);
    \draw[thin, orange] ( 5.5,0.5)--( 7.5,0.5);
    \draw[thin, orange] ( 7.5,0.5)--( 9.5,0.5);

    \draw[thin, orange] (11.5,0.5)--(13.5,0.5);
    \draw[thin, orange] (13.5,0.5)--(15.5,0.5);

    \draw[thin, red] ( 1,1.5)--( 3,1.5);
    \draw[thin, red] ( 3,1.5)--( 5,1.5);

    \draw[thin, red] ( 9,1.5)--(11,1.5);
    \draw[thin, red] (11,1.5)--(13,1.5);

    \draw[thin, violet] ( 0.5,1)--( 2.5,1);
    \draw[thin, violet] ( 2.5,1)--( 4.5,1);
    \draw[thin, violet] ( 4.5,1)--( 6.5,1);
    \draw[thin, violet] ( 6.5,1)--( 8.5,1);
    \draw[thin, violet] ( 8.5,1)--(10.5,1);
    \draw[thin, violet] (10.5,1)--(12.5,1);

        \foreach \x in {0, 2, 4, 6, 8, 10, 12, 14} {
        \node[blue] at (\x+1.5, 0.5) {$\times$};
        }

    \foreach \x in {0, 2, 4, 8, 10, 12} {
        \node[blue] at (\x+1, 1.5) {$\times$};
        }

    \foreach \x in {0, 2, 4, 6, 8, 10, 12,} {
        \node[blue] at (\x+0.5, 1) {$\times$};
        }

    \end{tikzpicture}
    \caption{Realization of a given sub-signature. In this picture is depicted how to realize the signature $\mu=(\,\mu_o,\mu_1,\mu_2\,)$, where $\mu_o=(\,4,2\,)$, $\mu_1=(\,2,2\,)$, and $\mu_2=(\,6\,)$. Notice that if $d=8$, then no triple sub-signatures with all even entries can be realized.}
    \label{fig:oddspintransurf4}
\end{figure}

\noindent The proof in this case is almost complete, and only two observations remain to be addressed. In the first place, we notice that at each step two different chains cannot overlap, because \eqref{eq:subsigncond} ensures that we have enough squares to make all of them disjoint. However, it is possible that two chains arising from different steps of the inductive foundation overlap. In fact, despite two subsets $Q_i$ and $Q_j$ are always disjoint, unless $i=j$, it can happen that the segment joining $y_i$ and $y_j$ is parallel to the existing chains. This issue can be simply avoided by taking the edges of all chains with a different rational slope. 

\smallskip

\noindent We finally verify that all translation surfaces that thus arise have odd spin parity. As a matter of our construction, for each handle we bubble step by step, there are two handle generators both of index one. As a direct consequence, the resulting structure has odd spin parity, because $(X,\omega)$, being a flat torus, has odd parity.

\smallskip

\subsubsection{Case ${\rm E}$.}\label{sssec:evencase} It remains to realize translation surfaces with even parity. Unlike the previous case, we cannot start from a flat torus or a surface of genus two, since there are no translation surfaces with even parity in these cases. Therefore, it is necessary to begin with a surface of genus three as the starting point for our construction. As in the odd case, let $\mu=(\mu_o,\dots,\mu_\ell)$ be a signature and let $d$ be the minimal degree allowed by $\mu$. Two remarks are in order.

\begin{rmk}
    Recall that if $\mu = (2m_1, \dots, 2m_k)$, then $g + 1 = m_1 + \cdots + m_k$. Applying the necessary condition \eqref{eq:subsigncond} to $\mu$, along with the assumption $g \geq 3$, it follows that $d \geq 4 + k \geq 5$. Moreover, we assume that $\mu$ has length at least two, since a translation surface with only one zero has no relative periods, and there would be nothing to prove. Therefore, with $k \geq 2$, we can further assume $d \geq 6$.
\end{rmk}

\begin{rmk}\label{rmk:stsperm}
    Every square-tiled surface is uniquely determined by a collection of $N$ unit squares and by a pair of permutations $(h,\,v)\in\mathfrak S_N\times\mathfrak S_N$ acting transitively on the set $\{\,1,\dots,N\,\}$. In fact, one can always label the squares from $1$ to $N$, and declare that $h(i)$, respectively $v(i)$, is the number of the neighbor to the right, respectively on the top, of the square $i$. The condition that $h$ and $v$ act transitively on $\{\,1,\dots,N\,\}$ is equivalent to the connectedness of the corresponding origami.
\end{rmk}

\noindent By making use of Remark \ref{rmk:stsperm}, for every $d\ge6$, we define $(X_d,\omega_d)$ as the square-tiled surface arising from the following permutations
\begin{equation}
    h\,=(\,1\,2\,\cdots\,d-4\,) \,(\,d-3\,d-2\,)\,(\,d-1\,d\,)\,\qquad v\,=\,(\,d-2\,d-1\,)\,(\,1\,d\,).
\end{equation}
For our purpose, it will be more convenient to realize such a surface as shown in Figure \ref{fig:basecaseevenspinsts}, which is obtained by gluing along a slit the flat torus $(Y_{d-4},\xi)$ arising from the cycle $h=\,(\,1\,2\,\cdots\,d-4\,)$ and the origami $\mathcal O$ arising from the cycles $h=\,(\,d-3\,d-2\,)\,(\,d-1\,d\,)$ and $v\,=\,(\,d-2\,d-1\,)$. It is an easy matter to check that these two square-tiled surfaces are isomorphic. 

\begin{figure}[h!]
    \centering
    \begin{tikzpicture}[scale=0.875, every node/.style={scale=0.825}]
    \definecolor{pallido}{RGB}{221,227,227}

    \pattern [pattern=north west lines, pattern color=pallido]
    (0,0)--(4,0)--(4,2)--(6,2)--(6,4)--(2,4)--(2,2)--(0,2)--(0,0);

    \pattern [pattern=north west lines, pattern color=pallido] (8,0)--(16,0)--(16,2)--(8,2)--(8,0);
    
    \draw[thin, black] (0,0)--( 4,0);
    \draw[thin, black] (8,0)--(16,0);
    
    \draw[thin, black]  (0,2)--( 4,2);
    \draw[thin, orange] (4,2)--( 6,2);
    \draw[thin, black]  (8,2)--(16,2);

    \draw[thin, black]  (2,4)--( 4,4);
    \draw[thin, orange] (4,4)--( 6,4);
    
    \foreach \x in {0, 2, 4, 8, 10, 12, 14, 16} {
        \draw (\x, 0) -- (\x, 2);
        }
    \foreach \x in {2, 4, 6} {
        \draw (\x, 2) -- (\x, 4);
        }

    \draw[orange] (09,1)--(11,1);

    \node at ( 1,  1)      {$5$};
    \node at ( 3,  1)      {$6$};
    \node at ( 3,  3)      {$7$};
    \node at ( 5,  3)      {$8$};
    \node at ( 9,2.5)      {$1$};
    \node at (11,2.5)      {$2$};
    \node at (13,2.5)      {$3$};
    \node at (15,2.5)      {$4$};

    \foreach \x in {8, 10, 12, 14, 16} {
        \fill (\x, 0) circle (2pt);
        }
    \foreach \x in {8, 10, 12, 14, 16} {
        \fill (\x, 2) circle (2pt);
        }
    \foreach \x in {0,2} {
        \node[blue] at (0,\x) {$\times$};
        }
    \foreach \x in {0,2,4} {
        \node[blue] at (4,\x) {$\times$};
        }
    \node[blue] at (9,1) {$\times$};
    
    \foreach \x in {0,2,4} {
        \node[red] at (2,\x) {$\times$};
        }
    \foreach \x in {2,4} {
        \node[red] at (6,\x) {$\times$};
        }
    \node[red] at (11,1) {$\times$};
    \end{tikzpicture}
    \caption{Realization of $(X_8,\omega)$: On the left the origami $\mathcal O$ and on the right the flat torus $(Y_4,\xi)$.}
    \label{fig:basecaseevenspinsts}
\end{figure}
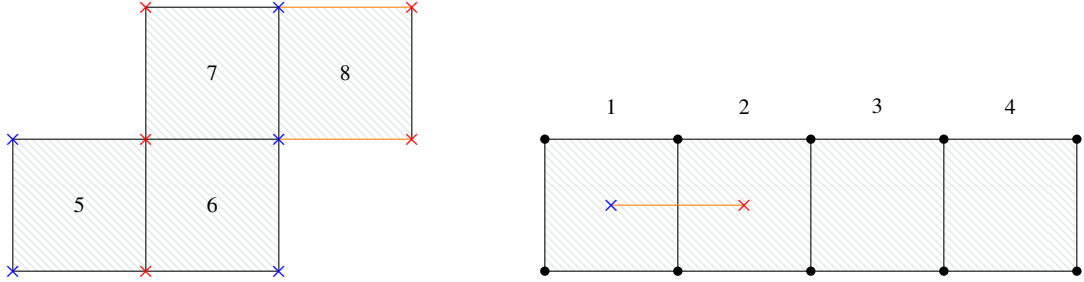

\noindent For every $d\ge6$, the resulting space is a square-tiled surface of genus three with two zeros of order two. Moreover, it can be directly seen that the spin parity is even. In particular, for $d=6$, we realize a structure with the minimal degree allowed by $\mu=(\,\mu_o,\,\mu_1\,)$ with $\mu_i=(\,2\,)$ for $i=0,1$. Notice that this is the only possible case for $g=3$. The following remark is in order.

\begin{rmk}
    For the reader's convenience, we recall that, according to the classification of the connected components of strata provided by Kontsevich and Zorich in \cite{KZ}, the stratum $\mathcal{H}_3(\,2,2\,)$ has exactly two connected components: one is hyperelliptic, and the other has odd spin parity. We observe that the structure just constructed is hyperelliptic. In fact, the argument above is an alternative of the one used in \S\ref{sssec:hypercomp} for $g=3$. For odd values of $g=2m+1$ greater than 3, the stratum $\mathcal H_g(\,2m,2m\,)$ has three connected components: two of these have even parity (the remaining one has odd parity), and these components with even parity are further distinguished by their hyperelliptic structure, one is hyperelliptic while the other is not. In the present section, we aim to construct translation surfaces with even parity that are not hyperelliptic — we have already discussed the construction of hyperelliptic structures in \S\ref{ssec:lastcasehyp}. Therefore, for these strata, it will be necessary to verify that the obtained structures do not admit any hyperelliptic involution. This verification is postponed to the end of the section.
\end{rmk}

\smallskip


\noindent Let us now assume $\mu=\mu_o=(\,2m_1,\dots,2m_k\,)$ and let $d$ be the minimal degree. We extend the previous construction as follows: We focus on the flat torus $(Y_{d-4},\xi)$ determined by the cycle $(\,1\,2\,\cdots\,d-4\,)$ and then proceed as in Case ${\rm O}$; see \S\ref{sssec:oddcase}. More precisely, $(Y_{d-4},\xi)$ comes with a genuine covering $\pi\colon Y_{d-4}\longrightarrow \mathbb T$ of degree $d-4$. Let $y_o\in\mathbb T$ be any point and let $Q_o=\{\,q_{o\,1},\dots,q_{o\,d-4}\,\}$ be the preimages of $y_o$ in $(Y_{d-4},\xi)$. We join two points if they are adjacent in the same sense as in \S\ref{sssec:oddcase}. In particular, we can assume that $\mathcal O$ is glued along the edge $e_1$ joining $q_{o1}$ and $q_{o2}$, thus making $y_o$ a branch value with two ramification points of order $2$. Depending on the signature $\mu_o$, the following mutually disjoint cases can occur.
\begin{itemize}
    \item[1.] $\mu_o=(2,2,2m_3,\dots,2m_k)$. Let $(X_d,\omega_d)$ be the square-tiled surface described above and consider a cylinder $C$ of $(Y_{d-4},\xi)$ of area $d-6$ disjoint from the slit -- \textit{e.g.} take the cylinder determined by the squares labeled from $3$ to $d-4$ glued as prescribed by the cycle $h=\,(\,1\,2\,\cdots\,d-4\,)$. Notice that, since $2m_1=2m_2=2$ and $\mu_o$ satisfy  \eqref{eq:subsigncond}, then
    \begin{equation}
        \sum_{i=3}^k (\,2m_i+1\,)\,\le\, d - 6.
    \end{equation}
    Thus we have enough room to realize $(2m_3,\dots,2m_k)$ in $C$ as done in \S\ref{sssec:oddcase}. Since $(X_d,\omega_d)$ has even parity and bubbling handles with zero volume does not alter it, the resulting structure also has even parity.
    \smallskip
    \item[2.] $\mu_o=(2,2m_2,\dots,2m_k)$. We can assume $m_i\ge2$ for all $i=2,\dots,k$; otherwise, we would reduce to the previous case. Once again, let $(X_d,\omega_d)$ be the square-tiled surface described above and let $Q_o$ be as above. Join $q_{oi}$ and $q_{oi+1}$ with an edge labeled as $e_i$ for $i=2,\dots,2m_2-1$, and orient all edges from the left to right. Slit all edges and label the resulting sides according to our convention and then identify $e_{2i}^{\pm}$ with $e_{2i+1}^{\mp}$ for all $i=1,\dots,m_2-1$. The resulting structure is a square-tiled surface and its abelian differential has signature $(\,2,\,2m_2\,)$ by construction, where the zero of order $2m_2$ arises from collapsing the points $q_{o2},\dots,q_{0,2m_2}$ of $(X_d,\omega_d)$. Notice that such a construction involves only $2m_2$ squares of $(Y_d,\xi)$ and there are other $d-4-2m_2$ untouched squares. Since $\mu_o$ satisfies \eqref{eq:subsigncond} and $2m_1=2$, we have 
    \begin{equation}
        \sum_{i=3}^k (\,2m_i+1\,)\le d - (2m_2 + 2 +2)=d-4-2m_2. 
    \end{equation}
    Therefore, we have enough room to realize the remaining part $(2m_3,\dots,2m_k)$ of $\mu_o$ without altering the spin parity.
    \smallskip
    
    \item[3.] $\mu_o=(2m_1,\dots,2m_k)$ with $k\ge2$ and assume $m_i\ge2$ for all $i=1,\dots,k$. To better understand how to handle this case, we first rewrite $(\,1\,2\,\cdots\,d-4\,)$ as $(\,d-(2m_k+1)\,\cdots\,d-4\,1\,2\,\cdots\,d-(2m_k+2)\,)$; see Figure \ref{fig:basecaseevenspinsts2}. Both cycles clearly determine $(Y_{d-4},\xi)$ and let $Q_o$ be as above. Join $q_{o\,i}$ and $q_{o\,i+1}$ with an edge labeled as $e_i$ for $i=2,\dots,2m_1-1$ and orient all edges from the left to right -- notice that these are $2m_1-2$ edges. Slit all these edges, label the resulting sides according to our convention, and then identify $e_{2i}^{\pm}$ with $e_{2i+1}^{\mp}$ for all $i=1,\dots,m_1-1$. In the same fashion, join $q_{o\,d-(2m_k+1)+j-1}$ and $q_{o\,d-(2m_k+1)+j}$ with an edge labeled as $e_{d-(2m_k+1)+j-1}$ where $j=1,\dots,2m_k-3$. It is straightforward to check that for $j=2m_k-3$, we get $q_{o\,d-4}$. Finally, join $q_{o\,d-4}$ and $q_{o\,1}$ with an edge labeled as $e_{d-4}$. We thus have another collection of $2m_k-2$ edges.
    Slit all these edges and re-glue the resulting sides as usual. The resulting structure is a square-tiled surface and its abelian differential has signature $(\,2m_1,\,2m_k\,)$. The zero of order $2m_1$ arises from collapsing the points $q_{o\,2},\dots,q_{o\,2m_1}$ of $(X_d,\omega_d)$ and the zero of order $2m_k$ arises from collapsing the points $q_{o1},\dots,q_{o\,d-(2m_k+1)}$. Notice that such a construction involves $2m_1+2m_k-2$ squares of $(Y_d,\xi)$ and there are other $(d-4)-(2m_1+2m_k-2)=d-(2m_1+2m_k+2)$ untouched squares. Since $\mu_o$ satisfies \eqref{eq:subsigncond}, then we have
    \begin{equation}
        \sum_{i=2}^{k-1} (\,2m_i+1\,)\le d - (2m_1 + 2m_k + 2).
    \end{equation}
    Therefore, we have enough room to realize the remaining part $(2m_2,\dots,2m_{k-1})$ of $\mu_o$ without altering the spin parity as done in \S\ref{sssec:oddcase}.
    \smallskip
\end{itemize}

\begin{figure}[h!]
    \centering
    \begin{tikzpicture}[scale=0.875, every node/.style={scale=0.825}]
    \definecolor{pallido}{RGB}{221,227,227}

    \pattern [pattern=north west lines, pattern color=pallido] (6,0)--(18,0)--(18,2)--(6,2)--(6,0);

    \pattern [pattern=north west lines, pattern color=pallido] (20,0)--(22,0)--(22,2)--(20,2)--(20,0);
    
    \draw[thin, black] (6,0)--(18,0);
    \draw[thin, black] (6,2)--(18,2);
    \draw[thin, black] (20,0)--(22,0);
    \draw[thin, black] (20,2)--(22,2);
    
    \foreach \x in {6, 8, 10, 12, 14, 16, 18, 20, 22} {
        \draw (\x, 0) -- (\x, 2);
        }

    \draw[orange] (11,1)--(13,1);
    \draw[blue]   ( 9,1)--(11,1);
    \draw[blue]   ( 7,1)--( 9,1);

    \draw[red]    (13,1)--(15,1);
    \draw[red]    (15,1)--(17,1);

    \node at ( 7,2.5)      {$d-5$};
    \node at ( 9,2.5)      {$d-4$};
    \node at (11,2.5)      {$1$};
    \node at (13,2.5)      {$2$};
    \node at (15,2.5)      {$3$};
    \node at (17,2.5)      {$4$};
    \node at (21,2.5)      {$d-6$};

    \node at (19 ,1) {$\cdots$};
    \node at ( 7.5,1.25) {$e_{d-5}$};
    \node at ( 9.5,1.25) {$e_{d-4}$};
    \node at (14.5,1.25) {$e_{2}$};
    \node at (16.5,1.25) {$e_{3}$};

    \foreach \x in {6, 8, 10, 12, 14, 16, 18, 20, 22} {
        \fill (\x, 0) circle (2pt);
        }
    \foreach \x in {6, 8, 10, 12, 14, 16, 18, 20, 22} {
        \fill (\x, 2) circle (2pt);
        }
     \node[red] at (13,1) {$\times$};
     \node[blue] at (11,1) {$\times$};
     \fill[black] ( 9,1) circle (2pt);
     \fill[black] ( 7,1) circle (2pt);
     \fill[black] (15,1) circle (2pt);
     \fill[black] (17,1) circle (2pt);

    \end{tikzpicture}
    \caption{Realization of the signature $(\,2m_1,\,2m_k\,)=(\,4,\,4\,)$. We have used only $6$ squares of $(Y_d,\xi)$ -- notice that $6=4+4-2$. Thus, we can use additional $d-8$ squares to realize $(2m_2,\dots,2m_{k-1})$.}
    \label{fig:basecaseevenspinsts2}
\end{figure}
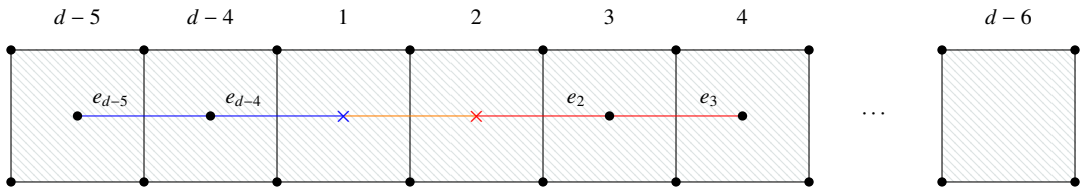

\begin{itemize}
     \item[4.] $\mu_o=(\,2m\,)$. We first observe that this case does not occur if $\ell=0$. However, if $\ell\ge1$, some sub-signatures can consist of a single zero, and we need to address this case to extend our constructions to partitioned signatures. This case easily follows by joining $q_{o\,d-4}$ and $q_{o1}$ with an edge $e_{d-4}$ and by joining $q_{o2}$ and $q_{o3}$ with an edge $e_{2}$. Slit these edges and identify $e_2^{\pm}$ with $e_{d-4}^{\mp}$. The resulting structure has a single zero of order $4$. We finally join $q_{0i}$ and $q_{0i+1}$ for $i=3,\dots,2m$ with edges labeled by $e_i$. We slit and re-glue them in the same manner as above. The resulting structure has a single zero of order $2m$ as desired and the parity remains unaltered. 
\end{itemize}

\smallskip

\noindent Next, we address the case $\ell\ge1$ which is based on the previous construction. Recall that $(X_d,\omega_d)$ comes with a branched covering map of degree $d$, say $\pi_d\colon X_d\longrightarrow \mathbb T$, which is branched at $q_{o1}$ and $q_{o2}$ by construction. Let $y_1\in\mathbb T$ such that $y_1-y_o=r_1$ is the prescribed relative period and let $Q_1=\{\,q_{11},\dots,q_{1d}\,\}$ be the preimages of $y_1$ via $\pi_d$. Since $d\ge6$, it follows that $Q_1$ has cardinality at least $6$. Consider the following points $q_{11},\,q_{12},\,q_{1d-3},\,\dots,\,q_{1d}$, where $q_{1i}$ belongs to the $i-$th square. Join $q_{11}$ and $q_{12}$ with an edge, say $a_1$. In the same fashion, join $q_{1d-3}$ and $q_{1d-2}$ with an edge $a_2$ and join $q_{1d-1}$ and $q_{1d}$ with an edge $a_3$. Upon orienting these edges from left to right, slit them and re-glue $a_i^+$ with $a_{i+1}^-$ for $i=1,2$ and $a_3^+$ with $a_1^-$. The resulting surface has two additional handles and a new pair of zeros of order $2$ that both project to $y_1$, making it a branch value. In other words, there is a newborn sub-signature $(\,2,\,2\,)$. Suppose we want to realize $\mu_1=(\,2m_{k+1},\dots,2m_{k+h}\,)$ for $h\ge1$. Then,  according to the structure of $\mu_1$, we return to one of the four cases listed above, and we can proceed in the same way, thus realizing $\mu_1$. An iterative argument shows that we can realize every signature $\mu=(\,\mu_o,\mu_1,\dots,\mu_\ell\,)$ for every $\ell\ge0$ with even parity as desired. 

\smallskip

\noindent To complete our discussion, it remains to verify that all structures realized in $\mathcal{H}_g(\,2m,2m\,)$ are not hyperelliptic. We recall that $(X_d, \omega_d)$ is obtained by gluing an origami $\mathcal O$ and the flat torus $(Y_{d-4}, \xi)$, and we observe that all the constructions performed above take place in this second component. To construct a structure in the stratum $\mathcal{H}_g(2m,2m)$, we proceed as in step 3 above. In this specific case, the construction described above is entirely equivalent to the following one. Pick a rectangle of size $(d-4) \times 1$, \textit{i.e.}, a flat torus of volume $d-4$, and then attach $g-3$ handles (since $g$ is odd, $g-3$ is even), as done in Section \S\ref{sssec:oddcase} for the odd case, to create a translation surface $(X, \omega)$ of genus $g-2$ and signature $(g-3, g-3)$ with odd parity and a saddle connection of unit length. Notice that this is always possible.  Next, we glue the origami $\mathcal{O}$ along the saddle connection, resulting in a translation surface of signature $(g-1, g-1)$ and even parity. Notice that $(X, \omega)$ has odd parity, so it cannot be hyperelliptic. Consequently, the resulting structure after gluing the origami $\mathcal{O}$ cannot be hyperelliptic either. Otherwise, by shrinking the slit to length zero, we would create a sequence of translation surfaces with hyperelliptic structures that degenerates to a hyperelliptic nodal differential, which requires both components to be hyperelliptic, leading to the desired contradiction. Said more generally, the closures of the hyperelliptic and non-hyperelliptic components remain disjoint in the multi-scale compactification of strata; see \cite{BCGGM2}. 

\begin{rmk}
    For strata of signature $(\,2m,2m\,)$, the same result is also obtained in \cite[Section \S3.2.2]{BJJP}. More precisely, one can pick $g=2m+1$ rectangles of size $2\,\times\,1$ and glue them along slits of unit length. The resulting structure has the desired properties.
\end{rmk}

\noindent This concludes the realization of signatures with prescribed parity in the holomorphic setting.

\bigskip

\bibliographystyle{amsalpha}
\bibliography{rprhdpi}

\end{document}